\documentclass[11pt,reqno]{amsart}

\usepackage{amsfonts, amsthm, amsmath}
\allowdisplaybreaks[4]

\usepackage{rotating}

\usepackage{tikz}

\usepackage{graphics}

\usepackage{amssymb}

\usepackage{amscd}

\usepackage[latin2]{inputenc}

\usepackage{t1enc}

\usepackage[mathscr]{eucal}

\usepackage{indentfirst}

\usepackage{graphicx}

\usepackage{graphics}

\usepackage{pict2e}

\usepackage{mathrsfs}

\usepackage{enumerate}
\usepackage[pagebackref]{hyperref}
\hypersetup{colorlinks=true}
\usepackage{cite}
\usepackage{color}
\usepackage{epic}
\usepackage{hyperref} %this gives clickable references, which is nice
\usepackage{framed}
\usepackage{mathabx}
\usepackage{booktabs}
\usepackage{makecell}
\newcolumntype{V}{!{\vrule width 2pt}}

\numberwithin{equation}{section}
\theoremstyle{plain}

\newtheorem{theorem}{Theorem}[section]

\newtheorem{corollary}[theorem]{Corollary}

\theoremstyle{definition}

\newtheorem{Def}[theorem]{Definition}

\newtheorem{remark}[theorem]{Remark}

\newtheorem{?}[theorem]{Problem}

\newcommand{\R}{\mathcal{R}}
\newcommand{\PP}{\mathcal{P}}
\newcommand{\s}{\mathcal{S}}
\newcommand{\B}{\mathcal{B}}
\newcommand{\T}{\mathcal{T}}
\newcommand{\I}{\mathfrak{B}}
\newcommand{\Rev}{\mathrm{Rev}}
\def\DT{\mathfrak{DT}}
\def\iar{\mathsf{iar}}
\def\lir{\mathsf{lir}}
\def\comp{\mathsf{comp}}
\def\DES{\mathrm{DES}}

\begin{document}

\title[Bijective Schr\"oder recurrences]{Bijective recurrences concerning two Schr\"oder triangles}

\author[S. Fu]{Shishuo Fu}
\address[Shishuo Fu]{College of Mathematics and Statistics, Chongqing University, Huxi campus, Chongqing 401331, P.R. China}
\email{fsshuo@cqu.edu.cn}

\author[Y. Wang]{Yaling Wang}

\address[Yaling Wang]{College of Mathematics and Statistics, Chongqing University, Huxi campus, Chongqing 401331, P.R. China}
\email{wyl032021@163.com}

\date{\today}

\begin{abstract}
Let $r(n,k)$ (resp. $s(n,k)$) be the number of Schr\"oder paths (resp. little Schr\"oder paths) of length $2n$ with $k$ hills, and set $r(0,0)=s(0,0)=1$. We bijectively establish the following recurrence relations:
\begin{align*}
r(n,0)&=\sum\limits_{j=0}^{n-1}2^{j}r(n-1,j),\\
r(n,k)&=r(n-1,k-1)+\sum\limits_{j=k}^{n-1}2^{j-k}r(n-1,j),\quad 1\le k\le n,\\
s(n,0) &=\sum\limits_{j=1}^{n-1}2\cdot3^{j-1}s(n-1,j),\\
s(n,k) &=s(n-1,k-1)+\sum\limits_{j=k+1}^{n-1}2\cdot3^{j-k-1}s(n-1,j),\quad 1\le k\le n.
\end{align*}
The infinite lower triangular matrices $[r(n,k)]_{n,k\ge 0}$ and $[s(n,k)]_{n,k\ge 0}$, whose row sums produce the large and little Schr\"oder numbers respectively, are two Riordan arrays of Bell type. Hence the above recurrences can also be deduced from their $A$- and $Z$-sequences characterizations.

On the other hand, it is well-known that the large Schr\"oder numbers also enumerate separable permutations. This propelled us to reveal the connection with a lesser-known permutation statistic, called initial ascending run, whose distribution on separable permutations is shown to be given by $[r(n,k)]_{n,k\ge 0}$ as well.
\end{abstract}

\subjclass[2010]{}


\maketitle

%\tableofcontents

%%%%%%%%%%%%%%%%%%%%%%%%%%%%%%%%%%%%%
\section{Introduction}\label{sec1: intro}
%%%%%%%%%%%%%%%%%%%%%%%%%%%%%%%%%%%%%

A \emph{Schr\"oder path} of length $2n$ is a path from $(0,0)$ to $(2n,0)$ consisting of steps $(x,y)\rightarrow(x+1,y+1)$ called \emph{ups}, steps $(x,y)\rightarrow(x+1,y-1)$ called \emph{downs}, and steps $(x,y)\rightarrow(x+2,y)$ called \emph{horizontals}, that never travels below the $x$-axis. The \emph{height} of a step is taken to be the $y$-coordinate of its ending point. A \emph{little Schr\"oder path} is a Schr\"oder path without horizontals at height $0$. When it is more convenient, we will describe a path by a word composed of letters $U$ (for an up), $D$ (for a down) and $H$ (for a horizontal). By convention, we view the empty path as the only (little) Schr\"oder path of length zero. We denote the set of all Schr\"oder paths (resp. little Schr\"oder paths) as $\R$ (resp. $\s$). These two types of lattice paths are well known to be enumerated by the large and little Schr\"oder numbers $r(n)$ and $s(n)$, respectively. They are registered in Sloane's OEIS \cite{oeis} as
\begin{align*}
r(n) [\href{http://oeis.org/A006318}{A006318}]:& 1,2,6,22,90,394,\cdots, \text{ and } \\
s(n) [\href{http://oeis.org/A001003}{A001003}]:& 1,1,3,11,45,197,\cdots.
\end{align*}

% The \emph{height} of a step is the $y$-coordinate of the ending point.

A \emph{peak} in a Schr\"oder path is an up followed by a down, such as $(x,y)\rightarrow(x+1,y+1)\rightarrow(x+2,y)$, which we refer to as a peak at height $y+1$. In particular, a \emph{hill} is a peak at height $1$. For $0\le k\le n$, let $r(n,k)$ (resp. $s(n,k)$) be the number of Schr\"oder paths (resp. little Schr\"oder paths) of length $2n$ with $k$ hills. Set $r(0,0)=s(0,0)=1$ by convention. Then we have
\begin{align}
[r(n,k)]_{n,k\ge 0} &=
\begin{bmatrix}\label{sch-tri}
1 & 0 & 0 & 0 & 0 & \cdots\\
1 & 1 & 0 & 0 & 0 & \cdots\\
3 & 2 & 1 & 0 & 0 & \cdots\\
11 & 7 & 3 & 1 & 0 & \cdots\\
45 & 28 & 12 & 4 & 1 & \cdots\\
\vdots & \vdots & \vdots & \vdots & \vdots & \ddots
\end{bmatrix} \quad [\href{http://oeis.org/A104219}{A104219}],\\
[s(n,k)]_{n,k\ge 0} &=
\begin{bmatrix}\label{littlesch-tri}
1 & 0 & 0 & 0 & 0 & \cdots\\
0 & 1 & 0 & 0 & 0 & \cdots\\
2 & 0 & 1 & 0 & 0 & \cdots\\
6 & 4 & 0 & 1 & 0 & \cdots\\
26 & 12 & 6 & 0 & 1 & \cdots\\
\vdots & \vdots & \vdots & \vdots & \vdots & \ddots
\end{bmatrix} \quad [\href{http://oeis.org/A114709}{A114709}].
\end{align}

Alternative interpretations of \eqref{sch-tri} were given by Pergola and Sulanke \cite{PS}, using weighted lattice paths and bicolored parallelogram polyominoes. Moreover, note that \eqref{sch-tri} and \eqref{littlesch-tri} are two examples of \emph{Riordan arrays}, which are natural generalizations of the Pascal triangle. Riordan arrays have attracted considerable interests recently, see for example Shapiro-Getu-Woan-Woodson \cite{SGWW} and Shapiro \cite{Sha} for the study of them as a group, called \emph{the Riordan group}, Sprugnoli \cite{Spr} and Luz\'on-Merlini-Mor\'on-Sprugnoli \cite{LMMS} for their application to combinatorial sums and identities, Barcucci-Pergola-Pinzani-Rinaldi \cite{BPPR} and Deutsch-Ferrari-Rinaldi \cite{DFR} for their relation with succession rule and the ECO method, Rogers \cite{Rog}, Merlini-Rogers-Sprugnoli-Verri \cite{MRSV} and He-Sprugnoli \cite{HS} for an alternative characterization of Riordan arrays using $A$- and $Z$-sequences, and most recently Chen-Liang-Wang \cite{CLW}, Zhu \cite{Zhu} and Chen-Wang \cite{CW} for the combinatorial inequalities in Riordan arrays. We mostly follow Barry \cite{Bar} for the notations and terminologies concerning Rirodan arrays. Readers unfamiliar with Riordan arrays could also use Barry's book as a gentle introduction to the subject.  

Towards the end of this introduction, we briefly recall how to apply the $A$- and $Z$-sequences characterizations of \eqref{sch-tri} and \eqref{littlesch-tri} to derive the following recurrences.

\begin{theorem}\label{thm:Sch}
For any integers $n\ge 1$, we have
\begin{align}
r(n,0)&=\sum\limits_{j=0}^{n-1}2^{j}r(n-1,j),\label{sch-Z}\\
r(n,k)&=r(n-1,k-1)+\sum\limits_{j=k}^{n-1}2^{j-k}r(n-1,j),\quad 1\le k\le n. \label{sch-A}
\end{align}
\end{theorem}

\begin{theorem}\label{thm:litSch}
We have $s(1,0)=0, s(1,1)=1$, and for any integers $n\ge 2$,
\begin{align}
s(n,0) &=\sum\limits_{j=1}^{n-1}2\cdot3^{j-1}s(n-1,j),\label{litsch-Z}\\
s(n,k) &=s(n-1,k-1)+\sum\limits_{j=k+1}^{n-1}2\cdot3^{j-k-1}s(n-1,j),\quad 1\le k\le n. \label{litsch-A}
\end{align}
\end{theorem}

M.-P. Sch\"utzenberger advocated that ``every algebraic relation is to be
given a combinatorial counterpart and vice versa'' (see \cite{FZ}). This is even more true for recurrences like \eqref{sch-Z}--\eqref{litsch-A} with positive coefficients. For instance, Foata and Zeilberger \cite{FZ} utilized well-weighted binary trees, and then Sulanke \cite{Sul} used elevated Schr\"oder paths to give bijective proofs of the following recurrence:
\begin{align*}
3(2n+1)r(n)=(n+2)r(n+1)+(n-1)r(n-1), \text{ for $n\ge 1$.}
\end{align*}

In Section~\ref{sec:comb pf}, we present bijective proofs, in the style of Foata-Zeilberger and Sulanke, of the above two theorems. Next in Section~\ref{sec:comb pf sep}, we consider \emph{separable permutations} (definition postponed to Section~\ref{sec:comb pf sep}), which are enumerated by the large Schr\"oder numbers as well. This raises the following Problem~\ref{find stat}. We succeeded in finding such a statistic, called \emph{initial ascending run} and denoted as $\iar$, that answers part (i) of this problem. The proof relies on interpreting \eqref{sch-Z} and \eqref{sch-A} via \emph{di-sk trees}, certain labeled binary trees introduced in \cite{FLZ} to represent separable permutations. We end this paper with the discussion on a more general Riordan array $(g(u,v;x),xg(u,v;x))$ and some of its specializations, as well as some outlook on future research motivated by the permutation statistic $\iar$.

\begin{?}\label{find stat}
\begin{enumerate}[(i)]
\item Is there a statistic defined on the set of separable permutations, such that the enumeration of separable permutations refined by this statistic is given by $r(n,k)$?
\item If yes, find a bijection from Schr\"oder paths to separable permutations, that sends the number of hills to the statistic found in (i).
\end{enumerate}
\end{?}

We close this introduction by recalling the typical algebraic proof of Theorems \ref{thm:Sch} and \ref{thm:litSch} using the theory of Riordan arrays. The readers are assumed to know the basic definitions. We first show that \eqref{sch-tri} and \eqref{littlesch-tri} are indeed Riordan arrays. Let $\widetilde{\R}$ be the set of Schr\"oder paths without hills. Then we have
$$|\s|=|\widetilde{\R}|,$$
since we can bijectively map a Schr\"oder path without hills to a little Schr\"oder path, by transforming each horizontal at height $0$ to a hill. Therefore we see $r(n,0)=s(n)$, and their common generating function is well-known to be

$$g(x):=\sum_{n\ge 0}r(n,0)x^n=\sum_{n\ge 0}s(n)x^n=\frac{1+x-\sqrt{1-6x+x^2}}{4x}.$$

Now each Schr\"oder path $p$ with $k$ hills can be uniquely decomposed as $p=p_1UDp_2\cdots p_kUDp_{k+1}$, where each $p_i,1\le i\le k+1$ is a Schr\"oder path without hills (possibly empty). This decomposition leads to the generating function

$$x^kg(x)^{k+1}=g(x)\cdot(xg(x))^k.$$

So we see $[r(n,k)]_{n,k\ge 0}=(g(x),xg(x))$ is indeed a Riordan array, and in particular, an element of the Bell subgroup \cite[\S 2.3]{Sha}. Therefore we can compute the generating functions for its $A$- and $Z$-sequences:
\begin{align*}
A(x)&=\frac{x}{\Rev(xg(x))}=\frac{1-x}{1-2x}=1+x+2x^2+2^2x^3+\cdots, \text{ and}\\
Z(x)&=\frac{1}{x}(A(x)-1)=\frac{1}{1-2x}=1+2x+2^2x^2+2^3x^3+\cdots,
\end{align*}
which implies Theorem~\ref{thm:Sch} immediately. Here $\Rev(f(x))$ is the compositional inverse of $f(x)$. For more details behind these calculations, see for example \cite[\S 7.5]{Bar}. The proof of Theorem~\ref{thm:litSch} using Riorday array requires the following generating function for little Sch\"oder paths without hills (see \href{http://oeis.org/A114710}{A114710}):
$$\sum_{n\ge 0}s(n,0)x^n=\frac{2}{1+3x+\sqrt{1-6x+x^2}}.$$
The rest can be done analogously, and thus is omitted.

%%%%%%%%%%%%%%%%%%%%%%%%%%%%%%%%%%%%%%%%
\section{Bijective proofs of Theorems~\ref{thm:Sch} and \ref{thm:litSch}}\label{sec:comb pf}
%%%%%%%%%%%%%%%%%%%%%%%%%%%%%%%%%%%%%%%%
We begin by introducing further notions defined on Schr\"oder paths.
\begin{Def}\label{further notion}
A \emph{$k$-basin at height $y$} is a down followed by $k$ consecutive horizontals at height $y$ and then an up, such as $$(x,y+1)\rightarrow(x+1,y)\rightarrow(x+3,y)\rightarrow \cdots \rightarrow (x+2k+1,y)\rightarrow(x+2k+2,y+1).$$ In particular, a $0$-basin is called a \emph{valley}. A \emph{closure} is a down step at height $0$. For $0\le k\le n$, let $\R_{n,k}$ (resp. $\s_{n,k}$) be the set of Schr\"oder paths (resp. little Schr\"oder paths) of length $2n$ with $k$ hills. Recall that their cardinalities are
$$|\R_{n,k}\vert=r(n,k), \text{ and } |\s_{n,k}\vert=s(n,k)$$
respectively. We denote the concatenation of two paths $p_1$ and $p_2$ by the juxtaposition $p_1p_2$, and it will be abbreviated as the power $p_1^2$ if $p_1=p_2$. We denote the set of binary sequences of length $k$ as
$$
\B_k:=\{(b_1,b_2,\ldots,b_k):\text{$b_i=0$ or $1$}, 1\le i\le k\}, \text{ for } k\ge 1, \text{ and } \B_0=\{\emptyset\}.
$$
\end{Def}

\begin{proof}[Proof of Theorem \ref{thm:Sch}] To prove \eqref{sch-Z}, we first define a map $$\phi:\bigcup_{k=0}^{n-1}\R_{n-1,k}\times\B_k\rightarrow \R_{n,0}.$$ Then it suffices to show that $\phi$ is a bijection.

Given any Schr\"oder path $p\in\R_{n-1,k}$ and any sequence $\mathbf{b}=(b_1,\dots,b_k)\in\B_k$, we transform the pair $(p,\mathbf{b})$ into a path $q\in\R_{n,0}$ by the four steps described below. Then we set $\phi(p,\mathbf{b})=q$.
\begin{description}
\item[Step 1] If $p$ is entirely composed of horizontals (so necessarily $k=0$), possibly empty, then we put $q=pH$.

\item[Step 2] Otherwise, we locate the leftmost up and the rightmost down of $p$, decompose it as $p=H^aUp_1DH^b$, for some integers $a,b\ge 0$ and subpath $Up_1D\in\R_{n-1-a-b,k}$. Let $$\hat{p}=H^aUUp_1DDH^b.$$ If $k=0$, then we put $q=\hat{p}$.
\item[Step 3] Otherwise we have $k\ge 1$. Label the $k$ hills in $Up_1D$ from left to right as $\mathrm{pk}_1,\ldots,\mathrm{pk}_k$. For each $j,1\le j\le k$, if $b_j=0$ we do nothing. If $b_j=1$ we consider the following cases:
\begin{itemize}
\item If $j=1$ and $\mathrm{pk}_1$ is at the beginning of $Up_1D$, i.e., $p_1$ begins with a $D$. Then we ``flatten'' this hill: $UD\rightarrow H$.
\item If $j=k$ and $\mathrm{pk}_k$ is at the end of $Up_1D$, i.e., $p_1$ ends with a $U$. Then we ``flatten'' this hill: $UD\rightarrow H$.
\item In all the remaining cases, we ``reverse'' the hill $\mathrm{pk}_j$: $UD\rightarrow DU$.
\end{itemize}
Denote the modified path as $\hat{q}$. If $\hat{q}$ has no hills, then we put $q=\hat{q}$.
\item[Step 4] Otherwise, clearly all hills in $\hat{q}$ were produced by ``reversing'' the hills in $Up_1D$ from Step 3. We fix this by replacing $m$ consecutive hills $D(UD)^mU$ with an $m$-basin $DH^mU$. Denote the modified path as $q$.
\end{description}
In all cases, we see $q$ indeed has length $2n$ and no hills, thus $q\in\R_{n,0}$ and $\phi$ is well-defined.

Next, we show that $\phi$ is bijective by constructing its inverse. For each path $q\in\R_{n,0}$, we determine an integer $k, 0\le k\le n-1$, find a path $p\in\R_{n-1,k}$ and a binary sequence $\mathbf{b}\in\B_k$, by following the three steps described below. Then we set $\phi^{-1}(q)=(p,\mathbf{b})$.
\begin{description}
	\item[Step 1] If $q$ is entirely composed of horizontals, then $k=0$, $\mathbf{b}=\emptyset$, and we get $p$ by removing an $H$ from $q$.
	\item[Step 2] Otherwise, we locate the leftmost up and the rightmost down of $q$, and decompose it as $q=H^aUq_1DH^b$. Count collectively how many initial horizontal (i.e., when $q_1$ begins with an $H$), final horizontal (i.e., when $q_1$ ends with an $H$), hills of $q_1$, and basins at height $0$ do we have in $Uq_1D$. The total number is denoted as $k$, with each $m$-basin contributing $m+1$ to the sum. Let $$\hat{q}=H^aq_1H^b.$$ If $k=0$, then put $\mathbf{b}=\emptyset$ and $p=\hat{q}$.
	\item[Step 3] Otherwise, we screen $q_1$ from left to right, and let the index $j$ increase from $1$ to $k$. Whenever we encounter a hill of $q_1$, we set $b_j=0$ and do nothing with the path, otherwise we consider the following cases:
\begin{itemize}
\item If we encounter an initial horizontal (i.e., $q_1$ begins with an $H$), then set $b_1=1$ and replace this $H$ with $UD$.
\item If we encounter a final horizontal (i.e., $q_1$ ends with an $H$), then set $b_k=1$ and replace this $H$ with $UD$.
\item If we encounter an $m$-basin at height $-1$ of $q_1$, then set $b_j=b_{j+1}=\cdots=b_{j+m}=1$, and replace this $m$-basin $DH^mU$ with $(UD)^{m+1}$. Note that $m=0$ is possible and allowed.
\end{itemize}
Denote the final path as $p$.
\end{description}

In all cases, we see $p\in\R_{n-1,k}$ and $\mathbf{b}\in\B_k$, hence $\phi^{-1}$ is well-defined. It is routine to check that $\phi$ and $\phi^{-1}$ are indeed inverse to each other. An example of applying $\phi^{-1}$ can be found in Figure \ref{inv-phi}, where those steps in $q$ that have contributed to the total number $k=7$ are labeled.

The proof of \eqref{sch-A} is analogous, we begin by constructing a map
$$\Phi:\bigcup_{j=k-1}^{n-1}\R_{n-1,j}\times\B_{j-k}\rightarrow \R_{n,k}, \text{ where }\B_{-1}:=\{\emptyset\}.$$
For each pair $(p,\mathbf{b})\in\bigcup_{j=k-1}^{n-1}\R_{n-1,j}\times\B_{j-k}$, we find a path $q\in\R_{n,k}$ according to the following two cases. Then we set $\Phi((p,\mathbf{b}))=q$.
\begin{description}
\item[Case 1] If $j=k-1$, then we put $q=UDp$.
\item[Case 2] Otherwise $j\ge k$. We decompose $p$ as $p=p_1UDp_2$, where $UD$ is the $k$-th hill of $p$ counting {\bf from right to left}. Therefore $p_1$ has $j-k$ hills. We see that $\phi(p_1,\mathbf{b})$ is a Schr\"oder path without hills. Now we put $$q=\phi(p_1,\mathbf{b})UDp_2.$$
\end{description}
In both cases, we see $q\in\R_{n,k}$ so $\Phi$ is well-defined. Moreover, one observes that $q$ is derived from Case 1 if and only if it begins with $UD$, making it clear how we should define the inverse of $\Phi$. Consequently we see that $\Phi$ is bijective. The proof is now completed.
\end{proof}

\begin{figure}[htb]
\setlength {\unitlength} {5.8mm}
\begin{picture} (30,8)
%%% q %%%%
\put(1,3){\line(1,0){2}}
\put(3,3){\line(1,0){2}}
\put(5,3){\line(1,1){1}}
\put(6,4){\line(1,0){2}}
\put(8,4){\line(1,1){1}}
\put(9,5){\line(1,0){2}}
\put(11,5){\line(1,-1){1}}
\put(12,4){\line(1,1){1}}
\put(13,5){\line(1,-1){1}}
\put(14,4){\line(1,-1){1}}
\put(15,3){\line(1,1){1}}
\put(16,4){\line(1,0){2}}
\put(18,4){\line(1,-1){1}}
\put(19,3){\line(1,0){2}}
\put(21,3){\line(1,0){2}}
\put(23,3){\line(1,1){1}}
\put(24,4){\line(1,1){1}}
\put(25,5){\line(1,-1){1}}
\put(26,4){\line(1,-1){1}}
\put(0,3){$q$:}
\put(6.8,4.1){$H$}
\put(12.1,4.6){$U$}
\put(13.5,4.6){$D$}
\put(24.1,4.6){$U$}
\put(25.5,4.6){$D$}
\put(13.8,3.1){$D$}
\put(15.6,3.1){$U$}
\put(17.8,3.1){$D$}
\put(19.8,3.1){$H$}
\put(21.8,3.1){$H$}
\put(23.6,3.1){$U$}
%%% p %%%
\put(1,0){\line(1,0){2}}
\put(3,0){\line(1,0){2}}
\put(5,0){\line(1,1){1}}
\put(6,1){\line(1,-1){1}}
\put(7,0){\line(1,1){1}}
\put(8,1){\line(1,0){2}}
\put(10,1){\line(1,-1){1}}
\put(11,0){\line(1,1){1}}
\put(12,1){\line(1,-1){1}}
\put(13,0){\line(1,1){1}}
\put(14,1){\line(1,-1){1}}
\put(15,0){\line(1,0){2}}
\put(17,0){\line(1,1){1}}
\put(18,1){\line(1,-1){1}}
\put(19,0){\line(1,1){1}}
\put(20,1){\line(1,-1){1}}
\put(21,0){\line(1,1){1}}
\put(22,1){\line(1,-1){1}}
\put(23,0){\line(1,1){1}}
\put(24,1){\line(1,-1){1}}
\put(0,0){$p$:}
\end{picture}
\caption{$\phi^{-1}(q)=(p,\mathbf{b})$ with $\mathbf{b}=(1,0,1,1,1,1,0)$.
\label{inv-phi}}
\end{figure}
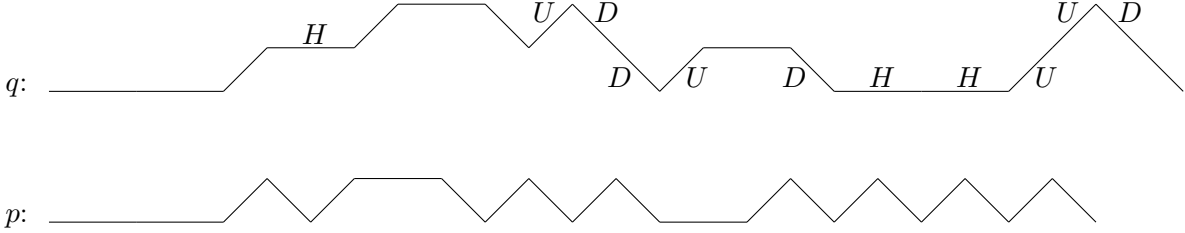

% \begin{example}\label{exam:sch}
% Take $R=HPQ$ be a element of $\mathfrak{R}_{14,6}$, where $P=UDUUDUDDUDUDUD$ and $Q=UDUUHHDDUD$ are generalized Schr\"oder paths with $4$ hills and $2$ hills, respectively.
% The bijection $\psi$ works on $R$ and produces $2^{4}$ Schr\"oder paths of semilength $15$ with $2$ hills in the following way: $\psi(R)=\varphi(HP)oQ=HU\tilde{P}DQ$, $\tilde{P}$ works as follows:

% \begin{table}[htb]
% \begin{tabular}{|c|c|c|}
% \hline
% $UDUUDUDDUDUDUD$ & $UDUUDUDDUDUDH$ & $UDUUDUDDUDDUUD$\\
% \hline
% $UDUUDUDDUDDUH$ & $UDUUDUDDDUUDUD$ & $UDUUDUDDDUUDH$\\
% \hline
% $UDUUDUDDDHUUD$ & $UDUUDUDDDHUH$ & $HUUDUDDUDUDUD$\\
% \hline
% $HUUDUDDUDUDH$ & $HUUDUDDUDDUUD$ & $HUUDUDDUDDUH$\\
% \hline
% $HUUDUDDDUUDUD$ & $HUUDUDDDUUDH$ & $HUUDUDDDHUUD$\\
% \hline
% $HUUDUDDDHUH$ & &\\
% \hline
% \end{tabular}
% \vspace{3mm}
% \caption{images of}
% \end{table}
% \end{example}

In order to deal with the factor $3$ appeared in equations \eqref{litsch-Z} and \eqref{litsch-A}, we need to introduce the following set of almost ternary sequences
$$
\T_k:=\{(t_1,t_2,\ldots,t_k):\text{$t_i=0,1$ or $2$}, 1\le i\le k-1, \text{ and } \text{$t_k=0$ or $1$} \}, \text{ for } k\ge 1.
$$

\begin{proof}[Proof of Theorem \ref{thm:litSch}] There is only one little Schr\"oder path of length $2$, namely $UD$, so $s(1,0)=0,s(1,1)=1$. Now suppose $n\ge 2$. To prove equation \eqref{litsch-Z}, it suffices to construct a bijection $$\psi:\bigcup_{k=1}^{n-1}\s_{n-1,k}\times\T_k\rightarrow \s_{n,0}.$$

Given any little Schr\"oder path $p\in\s_{n-1,k}$ and any sequence $\mathbf{t}=(t_1,\dots,t_k)\in\T_k$, we transform the pair $(p,\mathbf{t})$ into a path $q\in\s_{n,0}$ by the three steps described below. Then we set $\psi(p,\mathbf{t})=q$.

Since $k\geq 1$, we can uniquely decompose $p$ as $p=p_1UDp_2$, where $UD$ is the $k$-th hill of $p$ counting from left to right, $p_1$ is a little Schr\"oder path with $k-1$ hills, and $p_2$ is a little Schr\"oder path without hills. Label the $k$ hills in $p_1UD$ from left to right as $\mathrm{pk}_1,\ldots,\mathrm{pk}_k$.
\begin{description}
\item[Step 1] If $t_k=0$, let $$\hat{p}=UUp_1DDp_2.$$ For other $j,1\le j\le k-1$, if $t_j=0$ we do nothing. If $t_j\neq0$ we consider the following two cases:
\begin{itemize}
\item If $t_j=1$, then we ``flatten'' the hill $\mathrm{pk}_j$: $UD\rightarrow H$.
\item If $t_j=2$, then we ``reverse'' the hill $\mathrm{pk}_j$: $UD\rightarrow DU$.
\end{itemize}
Denote the modified path as $q$. In this case, there exist no horizontals at height $1$ to the left of the first closure (see Definition~\ref{further notion}), counting from left to right.
\item[Step 2] Otherwise we have $t_k=1$, let $$\hat{p}=UHp_1Dp_2.$$ For each $j,1\le j\le k-1$, we do the same operation as in Step 1, and denote the modified path as $\hat{q}$. If $\mathbf{t}$ has no $2$ in it, then we put $q=\hat{q}$. In this case, the $H$ after the first $U$ is the first horizontal at height $1$, and the $D$ before $p_2$ is the first closure.
\item[Step 3] Otherwise, each $t_j=2$ will create a new closure. We can uniquely decompose $\hat{q}$ as
$$\hat{q}=UHq_1DUq_2Dp_2,$$
where $D$ is the first closure, $q_{1}$ is a Schr\"oder path, and $Uq_2D$ is a little Schr\"oder path. Then we take $$q=UUq_2DHq_1Dp_2.$$
In this case, the $H$ before $q_1$ is the first horizontal at height $1$, and the $D$ before $p_2$ is the first closure.
\end{description}

In all cases, we see $q$ is a little Schr\"oder path of length $2n$ having no hills, thus $q\in\s_{n,0}$ and $\psi$ is well-defined.

Next, we construct the inverse of $\psi$. For each path $q\in\s_{n,0}$, we determine an integer $k, 1\le k\le n-1$, find a path $p\in\s_{n-1,k}$ and a sequence $\mathbf{t}\in\T_k$, by following the three steps described below. Then we set $\psi^{-1}(q)=(p,\mathbf{t})$.
\begin{description}
	\item[Step 1] If there exist no horizontals at height $1$ before the first closure, then we must have decompostition
	$$q=UUp_1DDp_2,$$ where the second $D$ is the first closure. Now for the subpath $p_1$, we screen from left to right, label collectively the hills, horizontals at height $0$, and valleys at height $-1$, as $\mathrm{phv}_1,\ldots,\mathrm{phv}_{k-1}$, and set $t_k=0$. For the extreme case when $p_1$ is empty or it has none of the steps mentioned above, we simply take $k=1$ and $\mathbf{t}=(0)$. Now for each $1\le j\le k-1$, we assign a value to $t_j$ and transform $\mathrm{phv}_j$ according to its type.
\begin{itemize}
\item If $\mathrm{phv}_j=UD$, we leave it unchanged and set $t_j=0$.
\item If $\mathrm{phv}_j=H$, we replace $H$ with $UD$ and set $t_j=1$.
\item If $\mathrm{phv}_j=DU$, we replace $DU$ with $UD$ and set $t_j=2$.
\end{itemize}
Denote the modified path as $\hat{p_1}$ and put $p=\hat{p_1}UDp_2$.
	\item[Step 2] Otherwise, locate the first horizontal at height $1$ (denoted as $H$), and the first closure (denoted as $D$), which must appear to the right of $H$, and decompose as
	$$q=Up_1Hq_1Dp_2,$$ If $p_1=\emptyset$, we screen $q_1$ from left to right, label collectively the hills and the horizontals at height $0$ as $\mathrm{ph}_1,\ldots,\mathrm{ph}_{k-1}$, and set $t_k=1$. For each $1\le j\le k-1$, we assign a value to $t_j$ and transform $\mathrm{ph}_j$ according to its type.
	\begin{itemize}
    \item If $\mathrm{ph}_j=UD$, we leave it unchanged and set $t_j=0$.
    \item If $\mathrm{ph}_j=H$, we replace $H$ with $UD$ and set $t_j=1$.
    \end{itemize}
    Denote the modified path as $\hat{q_1}$ and put $p=\hat{q_1}UDp_2$.
	\item[Step 3] Otherwise, $p_1$ is a non-empty little Schr\"oder path, so we can assume $p_1=Uq_2D$. We put $$\hat{q}=q_1DUq_2UDp_2.$$ Now label the hills, the horizontals at height $0$, and the valleys at height $-1$ in $q_1DUq_2$ as $\mathrm{phv}_1,\ldots,\mathrm{phv}_{k-1}$, and set $t_k=1$. For each $1\le j\le k-1$, we assign a value to $t_j$ and transform $\mathrm{phv}_j$ according to its type.
\begin{itemize}
\item If $\mathrm{phv}_j=UD$, we leave it unchanged and set $t_j=0$.
\item If $\mathrm{phv}_j=H$, we replace $H$ with $UD$ and set $t_j=1$.
\item If $\mathrm{phv}_j=DU$, we replace $DU$ with $UD$ and set $t_j=2$.
\end{itemize}
Denote the modified path as $p$.
\end{description}

In all cases, we see $p\in\s_{n-1,k}$ and $\mathbf{t}=(t_1,\ldots,t_k)\in\T_k$, hence $\psi^{-1}$ is well-defined. It is routine to check that $\psi$ and $\psi^{-1}$ are indeed inverse to each other. Two examples of applying $\psi^{-1}$ are given in Figures \ref{inv-psi1} and \ref{inv-psi2}, where those steps in $q$ that contribute to the total number $k=6$ are labeled. In addition, the first closure is marked by $\searrow$, while the first horizontal of height $1$ before the first closure is marked by $\rightarrow$.

The proof of \eqref{litsch-A} is analogous, we begin by constructing a map
$$\Psi:\bigcup_{j=k-1,j\neq k}^{n-1}\s_{n-1,j}\times\T_{j-k}\rightarrow \s_{n,k}, \text{ where }\T_{-1}=\{\emptyset\}.$$
For each pair $(p,\mathbf{t})\in\bigcup_{j=k-1,j\neq k}^{n-1}\s_{n-1,j}\times\T_{j-k}$, we construct a path $q\in\s_{n,k}$ according to the following two cases. Then we set $\Psi(p,\mathbf{t})=q$.
\begin{description}
\item[Case 1] If $j=k-1$, then we put $q=UDp$.
\item[Case 2] Otherwise $j> k$. We decompose $p$ as $p=p_1UDp_2$, where $UD$ is the $k$-th hill of $p$ counting {\bf from right to left}. Therefore $p_1$ has $j-k$ hills. We see that $\psi(p_1,\mathbf{t})$ is a little Schr\"oder path without hills. Now we put $$q=\psi(p_1,\mathbf{t})UDp_2.$$
\end{description}
In both cases, we see $q\in\s_{n,k}$ so $\Psi$ is well-defined. Moreover, one observes that $q$ is derived from Case 1 if and only if it begins with $UD$, making it clear how we should define the inverse of $\Psi$. Consequently we see that $\Psi$ is bijective. The proof is now completed.
\end{proof}

\begin{figure}[htb]
\setlength {\unitlength} {5.8mm}
\begin{picture} (30,8)
%%% q %%%%
\put(1,4){\line(1,1){2}}
\put(3,6){\line(1,0){2}}
\put(5,6){\line(1,1){1}}
\put(6,7){\line(1,1){1}}
\put(7,8){\line(1,-1){2}}
\put(9,6){\line(1,1){1}}
\put(10,7){\line(1,-1){2}}
\put(12,5){\line(1,1){1}}
\put(13,6){\line(1,-1){1}}
\put(14,5){\line(1,1){1}}
\put(15,6){\line(1,0){2}}
\put(17,6){\line(1,-1){2}}
\put(19,4){\line(1,1){2}}
\put(21,6){\line(1,-1){1}}
\put(22,5){\line(1,0){2}}
\put(24,5){\line(1,1){1}}
\put(25,6){\line(1,-1){2}}
\put(0,4){$q$:}
%%% label %%%
\put(3.8,6.2){$H$}
\put(9,6.4){$U$}
\put(10.6,6.4){$D$}
\put(11.4,5.7){$D$}
\put(12.3,5.7){$U$}
\put(13.4,5.7){$D$}
\put(14.3,5.7){$U$}
\put(15.8,6.2){$H$}
\put(18.3,4.5){$\searrow$}

%%% p %%%
\put(1,0){\line(1,1){1}}
\put(2,1){\line(1,-1){1}}
\put(3,0){\line(1,1){2}}
\put(5,2){\line(1,-1){2}}
\put(7,0){\line(1,1){1}}
\put(8,1){\line(1,-1){1}}
\put(9,0){\line(1,1){1}}
\put(10,1){\line(1,-1){1}}
\put(11,0){\line(1,1){1}}
\put(12,1){\line(1,-1){1}}
\put(13,0){\line(1,1){1}}
\put(14,1){\line(1,-1){1}}
\put(15,0){\line(1,1){1}}
\put(16,1){\line(1,-1){1}}
\put(17,0){\line(1,1){2}}
\put(19,2){\line(1,-1){1}}
\put(20,1){\line(1,0){2}}
\put(22,1){\line(1,1){1}}
\put(23,2){\line(1,-1){2}}
\put(0,0){$p$:}
\end{picture}
\caption{$\psi^{-1}(q)=(p,\mathbf{t})$ with $\mathbf{t}=(1,0,2,2,1,0)$.
\label{inv-psi1}}
\end{figure}
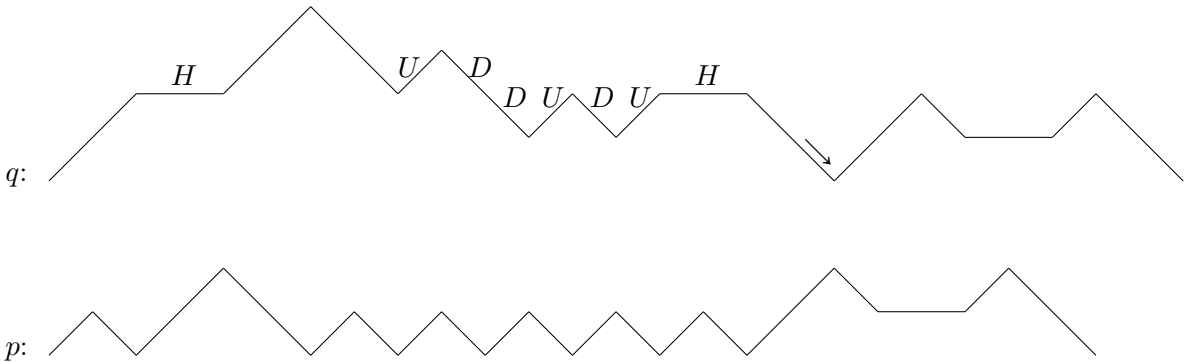

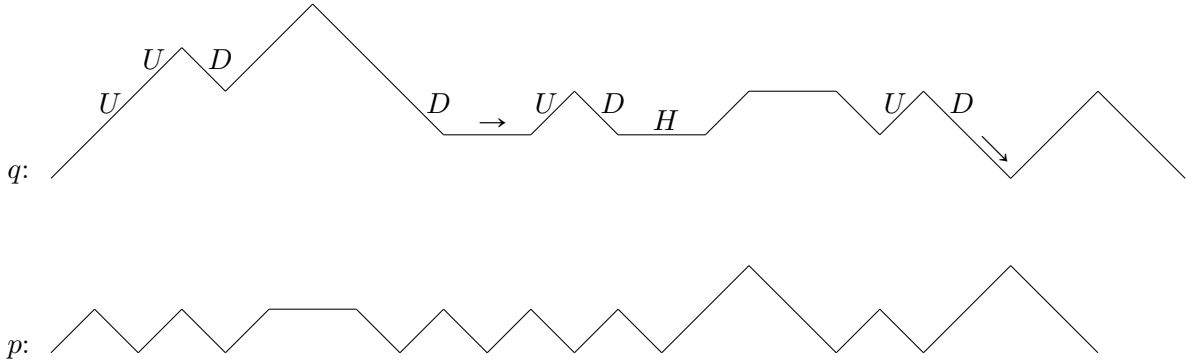
\begin{figure}
\setlength {\unitlength} {5.8mm}
\begin{picture} (30,8)
%%% q %%%%
\put(1,4){\line(1,1){3}}
\put(4,7){\line(1,-1){1}}
\put(5,6){\line(1,1){2}}
\put(7,8){\line(1,-1){3}}
\put(10,5){\line(1,0){2}}
\put(12,5){\line(1,1){1}}
\put(13,6){\line(1,-1){1}}
\put(14,5){\line(1,0){2}}
\put(16,5){\line(1,1){1}}
\put(17,6){\line(1,0){2}}
\put(19,6){\line(1,-1){1}}
\put(20,5){\line(1,1){1}}
\put(21,6){\line(1,-1){2}}
\put(23,4){\line(1,1){2}}
\put(25,6){\line(1,-1){2}}
\put(0,4){$q$:}
%%% label %%%
\put(22.3,4.5){$\searrow$}
\put(10.8,5.1){$\rightarrow$}
\put(3.1,6.5){$U$}
\put(4.6,6.5){$D$}
\put(12.1,5.5){$U$}
\put(13.6,5.5){$D$}
\put(14.8,5.1){$H$}
\put(20.1,5.5){$U$}
\put(21.6,5.5){$D$}
\put(2.1,5.5){$U$}
\put(9.6,5.5){$D$}

%%% p %%%
\put(1,0){\line(1,1){1}}
\put(2,1){\line(1,-1){1}}
\put(3,0){\line(1,1){1}}
\put(4,1){\line(1,-1){1}}
\put(5,0){\line(1,1){1}}
\put(6,1){\line(1,0){2}}
\put(8,1){\line(1,-1){1}}
\put(9,0){\line(1,1){1}}
\put(10,1){\line(1,-1){1}}
\put(11,0){\line(1,1){1}}
\put(12,1){\line(1,-1){1}}
\put(13,0){\line(1,1){1}}
\put(14,1){\line(1,-1){1}}
\put(15,0){\line(1,1){2}}
\put(17,2){\line(1,-1){2}}
\put(19,0){\line(1,1){1}}
\put(20,1){\line(1,-1){1}}
\put(21,0){\line(1,1){2}}
\put(23,2){\line(1,-1){2}}
\put(0,0){$p$:}
\end{picture}
\caption{$\psi^{-1}(q)=(p,\mathbf{t})$ with $\mathbf{t}=(0,1,0,2,0,1)$.
\label{inv-psi2}}
\end{figure}

% \begin{example}\label{exam:litsch}
% Take $S=PUDQM$ be a element of $s(14,6)$, where $P=UHUDHDUDUD$, $M=UDUDUUHHDDUD$ and $Q$ is empty set, are generalized little Schr\"oder paths with $2$ hills, $3$ hills and $0$ hill, respectively.
% The bijection $\tau$ works on $S$ and produces $2*3^{2}$ little Schr\"oder paths of semilength $15$ with $3$ hills in the following way:

% The last hill remains unchanged: $\tau(S)=\phi(PUDQ) o M=UU\tilde{P}DDQM$, $\tilde{P}$ works as follows:

% \begin{table}
% \begin{tabular}{|c|c|c|}
% \hline
% $UHUDHDUDUD$ & $UHUDHDUDH$ & $UHUDHDUDDU$\\
% \hline
% $UHUDHDHUD$ & $UHUDHDHH$ & $UHUDHDHDU$\\
% \hline
% $UHUDHDDUUD$ & $UHUDHDDUH$ & $UHUDHDDUDU$\\
% \hline
% \hline
% \end{tabular}
% \end{table}

% Replacing the last hill with a horizontal step: $\tau(S)=\phi(PUDQ) o M=U\tilde{P}DQM$, $\tilde{P}$ works as follows:
% \begin{table}
% \begin{tabular}{|c|c|c|}
% \hline
% $HUHUDHDUDUD$ & $HUHUDHDUDH$ & $UDHUHUDHDUD$\\
% \hline
% $HUHUDHDHUD$ & $HUHUDHDHH$ & $UDHUHUDHDH$\\
% \hline
% $UUDDHUHUDHD$ & $UHDHUHUDHD$ & $UDUDHUHUDHD$\\
% \hline
% \hline
% \end{tabular}
% \end{table}
% \end{example}

%%%%%%%%%%%%%%%%%%%%%
\section{Separable permutations}\label{sec:comb pf sep}
%%%%%%%%%%%%%%%%%%%%%

Other than the Schr\"oder paths, there are quite a few combinatorial structures enumerated by the large Schr\"oder numbers (see for example \cite[Exercise 6.39]{StaEC2}). One of them is the set of \emph{separable permutations} (see \cite{SS,Wes,Kit}). A permutation is called \emph{separable}, if it does not contain a subsequence of four elements with the same pairwise comparisons as $2413$ or $3142$. We denote the set of all separable permutations of $[n]:=\{1,2,\ldots,n\}$ by $\mathfrak{S}_{n}(2413,3142)$, where $\mathfrak{S}_n$ is the $n$-th symmetric group. Now we define $\iar$ on the entire symmetric group $\mathfrak{S}_n$. It should be noted that in the literature, $\iar$ has made its appearance as $\lir$ (leftmost increasing run), see \cite[pp. 5]{CK}.

\begin{Def}
For any permutation $\pi=\pi_1\pi_2\cdots\pi_n\in\mathfrak{S}_n$, we take $\iar(\pi)$ to be the greatest integer $i, 1\le i\le n$ such that $\pi_1<\cdots<\pi_{i}$. For each $i, 1\le i< n$, we call it a \emph{descent} of $\pi$, if and only if $\pi_i>\pi_{i+1}$. Then alternatively, $\iar(\pi)$ is the smallest descent of $\pi$. And $\iar(\pi)=n$ if and only if $\pi=12\cdots n$ has no descents at all.
\end{Def}

Now for $1\le k\le n$, let $$\PP_{n,k}:=\{\pi\in\mathfrak{S}_n(2413,3142):\iar(\pi)=k\},$$ whose cardinality we denote as $p(n,k)$. Then we have the following theorem.
\begin{theorem}\label{Sep-iar}
We have $p(1,1)=1$, and
\begin{align}
p(n,1) &=\sum_{j=1}^{n-1}2^{j-1}p(n-1,j), \; \text{ for }n\ge 2,\label{sep-Z}\\
p(n,k) &=p(n-1,k-1)+\sum_{j=k}^{n-1}2^{j-k}p(n-1,j),\; \text{ for }2\le k \le n.\label{sep-A}
\end{align}
Consequently, for all $1\le k\le n$,
\begin{align}
r(n-1,k-1)=p(n,k).\label{p=r}
\end{align}
\end{theorem}
\begin{remark}
It is unclear to us, how the initial ascents $\pi_1<\pi_2<\cdots<\pi_{\iar(\pi)}$ can decompose the permutation $\pi$ as the hills does to a Schr\"oder path. Therefore it seems difficult, if not impossible, to prove Theorem~\ref{Sep-iar} algebraically using Riorday array.
\end{remark}
Equation \eqref{p=r} immediately follows from the same recurrences \eqref{sch-Z}--\eqref{sch-A}, and \eqref{sep-Z}--\eqref{sep-A}, as well as the fact that $r(0,0)=p(1,1)=1$. Moreover, it justifies using $\iar$ as a valid statistic to answer Problem~\ref{find stat} (i). All it remains is to show \eqref{sep-Z} and \eqref{sep-A}. To this end, we need to utilize certain kind of labeled and rooted binary trees, which was formally defined in \cite{FLZ}. And implicitly, it had previously appeared in Shapiro and Stephens' work \cite{SS}.
\begin{Def}{\cite[Def.~2.2]{FLZ}}
A rooted binary tree is called \emph{di-sk tree} if its nodes are labeled either with $\oplus$ or $\ominus$, such that no node has the same label as its right child (this is called the \emph{right chain condition}). We use the \emph{in-order} (traversal) to compare nodes on di-sk trees: starting with the root node, we recursively traverse the left subtree to parent then to the right subtree if any. The set of all di-sk trees with $n-1$ nodes is denoted as $\mathfrak{DT}_{n}$. The $i$-th node of a di-sk tree $T$ is denoted as $T(i)$.
\end{Def}
\begin{remark}
For each di-sk tree $T\in\mathfrak{DT}_n$, it uniquely decides a \emph{label sequence} $(L_1,L_2,\ldots,L_{n-1})$, where $L_i$ is the label of the $i$-th node (by in-order) of $T$. But conversely this is not true. See Fig.~\ref{label seq} for all four di-sk trees in $\mathfrak{DT}_{4}$ that share the same label sequence. Note that the last tree is not a di-sk tree since it does not satisfy the right chain condition.
\end{remark}

\begin{figure}
\begin{tikzpicture}[scale=0.4]
%% edges
\draw[-] (1,1) to (3,3);
\draw[-] (4,4) to (6,6);

\draw[-] (10,3) to (12,1);
\draw[-] (10,4) to (12,6);

\draw[-] (16,1) to (18,3);
\draw[-] (19,3) to (21,1);

\draw[-] (25,6) to (27,4);
\draw[-] (28,3) to (30,1);

\draw[-] (34,6) to (36,4);
\draw[-] (36,3) to (34,1);
%% labels
\node at (.5,.5) {$\ominus$};
\node at (3.5,3.5) {$\oplus$};
\node at (6.5,6.5) {$\ominus$};

\node at (9.5,3.5) {$\ominus$};
\node at (12.5,.5) {$\oplus$};
\node at (12.5,6.5) {$\ominus$};

\node at (15.5,.5) {$\ominus$};
\node at (18.5,3.5) {$\oplus$};
\node at (21.5,.5) {$\ominus$};

\node at (24.5,6.5) {$\ominus$};
\node at (27.5,3.5) {$\oplus$};
\node at (30.5,0.5) {$\ominus$};

\node at (33.5,6.5) {$\ominus$};
\node at (33.5,.5) {$\oplus$};
\node at (36.5,3.5) {$\ominus$};
\end{tikzpicture}
\caption{Five trees with label sequence $(\ominus,\oplus,\ominus)$, only the first four being di-sk trees.}\label{label seq}
\end{figure}
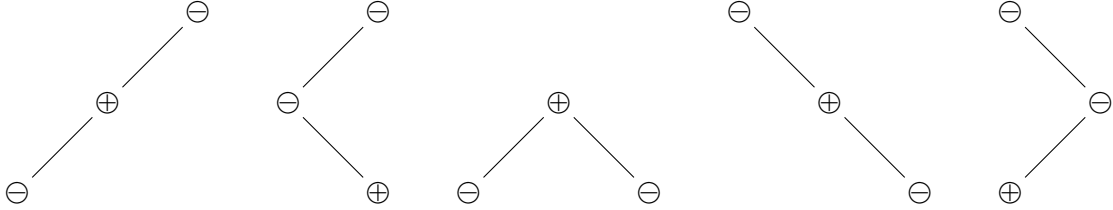

A di-sk tree is called \emph{right-branching}, if its root does not have left child. For example, only the fourth tree in Fig~\ref{label seq} is right-branching. We need the following theorem, its corollary, and three operations.

\begin{theorem}[Theorem 2.3 in \cite{FLZ}]\label{FLZ-bij}
There exists a bijection $\eta: \mathfrak{S}_n(2413,3142)\rightarrow \mathfrak{DT}_n$ such that
\begin{align*}
i\in\DES(\pi)\; &\Leftrightarrow \; \text{the i-th node (by in-order) of $\eta(\pi)$ is labeled $\ominus$},
\end{align*}
where $\DES(\pi)$ is the set of all descents of $\pi$.
\end{theorem}
This map $\eta$ bijectively establishes the following corollary.
\begin{corollary}\label{dt=sp/iar}
For any $1\le k\le n$, let $\DT_{n,k}$ denote the set of all di-sk trees with $n-1$ nodes whose label sequence begins as $(\underbrace{\oplus,\ldots,\oplus}_{k-1},\ominus,\ldots)$. Then we have $$|\DT_{n,k}|=p(n,k).$$
\end{corollary}

\begin{Def}
Given two di-sk trees $S\in\mathfrak{DT}_m$ and $T\in\mathfrak{DT}_n$, we define three operations, each of which combines $S$ and $T$ to give us a di-sk tree in $\mathfrak{DT}_{m+n-1}$.
\begin{itemize}
	\item For any node of $T$ that has no left child, say $T(i)$, let $$S\slash T[i]$$ be the di-sk tree derived by attaching $S$ to $T$ from left, such that the root of $S$ becomes the {\bf left} child of $T(i)$.
	\item If the root of $S$ is labeled $\oplus$ (resp. $\ominus$), then for any $\ominus$-node (resp. $\oplus$-node) of $T$ that has no right child, say $T(i)$, we let $$T[i]\backslash S $$ be the di-sk tree derived by attaching $S$ to $T$ from right, such that the root of $S$ becomes the {\bf right} child of $T(i)$.
	\item Assume both $S$ and $T$ are right-branching. If $S$ is a single node, then let $$S\#T:=S\slash T[1],$$ otherwise we set $$S\#T:=T\slash S[2].$$ In both cases, we see that $S(1)$ and $T(1)$ are the first and second node of $S\#T$.
\end{itemize}
\end{Def}
The notion of \emph{right chain} is of great importance due to the right chain condition for all di-sk trees. We introduce it now, together with an involution that acts on each di-sk tree.
\begin{Def}
Given a di-sk tree, its right chain (or simply chain) is any maximal chain composed of only right edges. We say two nodes are at \emph{the same level} if they are connected by a sequence of left edges. Then, whether two chains are at the same level or not is according to their heads.
\end{Def}
\begin{Def}
For each di-sk tree $T$, the map $\tau$ switches the label of the root of $T$ (i.e., $\ominus$ becomes $\oplus$ and $\oplus$ becomes $\ominus$), and thus switches the labels of all the nodes in the same chain as the root, to guarantee that $\tau(T)$ is still a di-sk tree. Moreover, it keeps the labels of all the remaining nodes in $T$ and the tree structure of $T$.
\end{Def}

We take $\tau^0(T)=T$ and $\tau$ is clearly an involution on $\DT_n$. Let $\I_0:=\{0\}$ and for $k\ge 1$,
$$\I_k:=\{b=b_1\cdots b_k: \text{$b_i=0$ or $1$, for $1\le i\le k$}\}$$
be the set of binary numbers with exactly $k$ digits, whose cardinality is $2^k$. We are now ready to prove the main result of this section.

% By Proposition, for every separable permutation $\pi\in \PP_{n,k}$, we can find the unique di-sk tree $\sigma(\pi)\in \mathfrak{DT}_n$ satisfies that the $k$th node of $\sigma(\pi)$ is labeled $\ominus$ and the $j$th node of $\sigma(\pi)$ is labeled $\oplus$, where $1\leq j < k$. Denote by $\mathfrak{DT}_{n,k}$ the set of di-sk trees satisfying the above property, and $d(n,k)$ denotes the number of $\mathfrak{DT}_{n,k}$. In particular, a di-sk tree $a\in \mathfrak{DT}_{n,n}$ is unique and every node of $a$ is labeled with $\oplus$.

% In the following section, we will prove that $d(n,k)$ satisfies the equation
% \begin{align}
% d(n,k)=d(n-1,k-1)+\sum\limits_{j=k}^{n-1}2^{j-k}d(n-1,j),\quad 1\leq k \leq n
% \end{align}
% which completes the proof of Theorem \ref{Sep-iar}.

% For $k$, we denote the index set and the set of binary sequence of length $i$ as
% $$
% \I_j:=\{i:0\leq i<j\}, \text{ for } j\ge 0, \text{ and } \I_{-1}={-1}.
% $$
% $$
% \B_{j,i}:=\{(b_0,b_1,\ldots,b_{i-1}):\text{$b_t=0$ or $1$}, 0\le t\le i-1\}, \text{ for } i\ge 1, \text{ and }
% $$
% $$
% \B_{k-1,-1}=\B_{k,0}=\{\emptyset\}, \B_{j,0}=\{(b_0):\text{$b_0=0$ or $1$}\}, \text{ for } j\ge k.
% $$
\begin{proof}[Proof of Theorem \ref{Sep-iar}]First note that $\mathfrak{DT}_{n-1,k-1}$ is in bijection with the following subset of $\mathfrak{DT}_{n,k}$.
$$\mathfrak{DT}_{n,k}^*:=\{T\in\DT_{n,k}: \text{there exists $S\in\DT_{n-1,k-1}$ such that $T=\oplus\slash S[1]$}\},$$
which is the empty set if $k=1$. In view of Corollary~\ref{dt=sp/iar}, it then suffices to construct a bijection $$\rho:\bigcup_{j=k}^{n-1}\DT_{n-1,j}\times\I_{j-k}\rightarrow \DT_{n,k}\setminus \DT_{n,k}^*, \text{ for $1\le k\le n-1$}.$$

% Now we describe a bijective way to complete the proof. First, we need to define the mapping
% $$\rho:\bigcup_{j=k-1}^{n-1}\mathfrak{DT}_{n-1,l}\times\I_{l-k}\times\B_{l,\I_{l-k}}\rightarrow \mathfrak{DT}_{n,k}.$$ Then it suffices to show that $\rho$ is a bijection.

For each pair $(T,b)\in\DT_{n-1,j}\times\I_{j-k}$, we construct a di-sk tree $S\in\mathfrak{DT}_{n,k}\setminus\DT_{n,k}^*$ by three steps described below. Then we set $\rho(T,b)=S$. Note that since $j\ge k$, the label sequence of $T$ must begins as $(\underbrace{\oplus,\ldots,\oplus}_{k-1},L_k,\ldots,L_{n-2})$.

\begin{description}
\item[Step 1] If $T(k)$ exists and has no left child, then let $$\hat{T}:=\ominus\slash T[k].$$ Otherwise $T(k)$ does not exist (this happens only when $k=n-1$), or the left child of $T(k)$ is $T(k-1)$, which is labeled $\oplus$. We let $$\hat{T}:=T[k-1]\backslash\ominus.$$ In either case, we see $\hat{T}\in\DT_{n,k}$ and the newly added $\ominus$-node is $\hat{T}(k)$. Now if $\hat{T}\in\DT_{n,k}\setminus \DT_{n,k}^*$, let $\hat{S}:=\hat{T}$ and go to Step 3.

\item[Step 2] Otherwise we have $\hat{T}\in\DT_{n,k}^*$, so $\hat{T}(1)$ is labeled $\oplus$ and does not have right child. Suppose $\hat{T}(m)$ is the smallest indexed node at the same level as $\hat{T}(1)$ that has right child. Clearly $2\le m \le k-1$. We cut the edge between $\hat{T}(m-1)$ and $\hat{T}(m)$ to get two subtrees. The one rooted at $\hat{T}(m-1)$ is denote by $R$, which is a path from $\hat{T}(1)$ to $\hat{T}(m-1)$ composed of left edges only. The other subtree is denoted as $\tilde{T}$. Now we put $$\hat{S}:=R\slash\tilde{T}[k-m+1].$$ One checks that $\hat{S}\in\DT_{n,k}\setminus \DT_{n,k}^*$. More precisely, $\hat{S}(k)$ and $\hat{S}(j+1)$ are labeled $\ominus$, and $\hat{S}(1),\ldots,\hat{S}(k-1),\hat{S}(k+1),\ldots,\hat{S}(j)$ are all labeled $\oplus$.

\item[Step 3] If $b$ contains no $1$, then take $S:=\hat{S}$ and we are done. Otherwise suppose $$b=0\cdots 01\hat{b},$$ where $\hat{b}=\hat{b}_1\cdots\hat{b}_{l}$, for some $0\le l\le j-k-1$. To get $S$, we ``cut and paste'' as follows.
\begin{description}
\item[Step 3-1] We delete the southwest and northeast edges $$
    \begin{smallmatrix}
    & & \diagup \\
	& \oplus & \\
	\diagup & & \\
	\end{smallmatrix},$$ if any, of all nodes $\hat{S}(k+1),\ldots,\hat{S}(k+l+1)$, and remove them as well as their descendants from $\hat{S}$. If doing so creates a pair of disconnected nodes that were at the same level in $\hat{S}$, then connect them. This modified tree still contains $\hat{S}(k)$ and is denoted as $\tilde{S}$. All the $l+1$ ``peeled off'' trees are right-branching, and we denote the one rooted at $\hat{S}(k+i)$ as $\hat{S}_{i}$, for $1\le i\le l+1$. 
\item[Step 3-2] We take
$$S:=\tilde{S}[k]\backslash P,$$ where $$P:=\tau^{\hat{b}_1}(\hat{S}_{1})[1]\slash\tau^{\hat{b}_2}(\hat{S}_2)[1]\slash\cdots\slash\tau^{\hat{b}_l}(\hat{S}_l)[1]\slash\hat{S}_{l+1}[1].$$ Intuitively, $P$ is formed by connecting the roots of $\tau^{\hat{b}_1}(\hat{S}_1),\ldots,\tau^{\hat{b}_l}(\hat{S}_l),\hat{S}_{l+1}$ one-by-one using left edges.
\end{description}
\end{description}

In all cases, we see $S\in\mathfrak{DT}_{n,k}$ and $\rho$ is well-defined. See Figure \ref{di-sk--1} for an example of applying $\rho$ on pairs $(T,b)$, with a fixed $T\in\DT_{14,6}$ and all possible choices of $b\in\I_4$.

Next, we show that $\rho$ is bijective by constructing its inverse. For each di-sk tree $S\in\DT_{n,k}\setminus \DT_{n,k}^*$, we determine an integer $j, 1\le j\le n-1$, construct a tree $T\in\mathfrak{DT}_{n-1,j}$, and a binary number $b\in\I_{j-k}$, by following the three steps described below. Then we set $\rho^{-1}(S)=(T,b)$.

\begin{description}
	\item[Step 1] Basically, we want to undo the ``cut and paste'' to recover $\hat{S}$, from which we can recognize the integer $j$. 
    \begin{description}
    	\item[Step 1-1] If $S(k)$ has no right child, then we put $\hat{S}=S$, $b=\underbrace{0\cdots 0}_{j-k}$ ($j$ will be determined in Step 2) and go to Step 2. Otherwise suppose $S(i_1),\ldots,S(i_l),S(i_{l+1})$ are all the nodes at the same level as the right child of $S(k)$, with $S(i_{l+1})$ being the right child itself. Let $$\hat{b}=\hat{b}_1\cdots\hat{b}_l,$$ where $$\hat{b}_m=\begin{cases}0 & \text{if $S(i_{m})$ is labeled $\oplus$}\\ 1 & \text{if $S(i_{m})$ is labeled $\ominus$}\end{cases}, \text{ for $1\le m\le l$}.$$ Now we cut the edge between $S(k)$ and $S(i_{l+1})$, as well as the southwest and northeast edges $$
    \begin{smallmatrix}
    & & \diagup \\
	& \oplus & \\
	\diagup & & \\
	\end{smallmatrix}$$ of all nodes $S(i_1),\ldots,S(i_{l})$, denote the subtree rooted at $S(i_m)$ as $S_m$, for $1\le m\le l+1$. The remaining subtree contains $S(k)$ and is denoted as $\tilde{S}$. Note that $S(k)=\tilde{S}(k)$ and all $S_m, 1\le m\le l+1$ are right-branching (see the definition before Theorem~\ref{FLZ-bij}). 
    	\item[Step 1-2] We get $\hat{S}$ by inserting $P$ into $\tilde{S}$ according to the following two cases, where
        $$P:=\tau^{\hat{b}_1}(S_1)\#(\tau^{\hat{b}_2}(S_2)\#(\cdots(\tau^{\hat{b}_l}(S_l)\#S_{l+1}))).$$
        \begin{itemize}
        	\item If $\tilde{S}(k)$ has a left parent, say $\tilde{S}(k_1)$, then connect $\tilde{S}(k_1)$ with $P(1)$ using a left edge. Moreover if $\tilde{S}(k_1)$ has a right parent in $\tilde{S}$, say $\tilde{S}(k_2)$, then connect the root of $P$ with $\tilde{S}(k_2)$ using a left edge. Denote this new di-sk tree by $\hat{S}$.
        	\item Otherwise we connect $\tilde{S}(k)$ with $P(1)$ using a left edge. Moreover if $\tilde{S}(k)$ has a right parent in $\tilde{S}$, say $\tilde{S}(k_3)$, then connect the root of $P$ with $\tilde{S}(k_3)$ using a left edge. Denote this new di-sk tree by $\hat{S}$.
        \end{itemize}
    \end{description}
    \item[Step 2] Clearly $\hat{S}(k)$ is still the first $\ominus$-node in $\hat{S}$, now suppose the second $\ominus$-node in $\hat{S}$ is the $m$-th node, then take $j=m-1$. In the case that $\hat{S}$ has only one $\ominus$-node, we take $j=n-1$. If $b$ is undefined, we put $$b:=\underbrace{0\cdots 0}_{j-k-l-1}1\hat{b}.$$ Now if $\hat{S}(k)$ does not have left child, then let $\hat{T}=\hat{S}$ and go to Step 3. Otherwise, we cut the edge between $\hat{S}(k-1)$ and $\hat{S}(k)$ to have two subtrees. The one rooted at $\hat{S}(k-1)$ is denoted as $R$, the other one has $\hat{S}(k)$ and is denoted as $\tilde{T}$. Now we put $$\hat{T}:=R\slash\tilde{T}[1].$$
    \item[Step 3] It should be clear that after the above two steps, $\hat{T}(k)$ is a leaf node (i.e. has no children) in $\hat{T}$. We simply delete this node and its associated edge from $\hat{T}$, and denote the new di-sk tree as $T$.
\end{description}

One could check that $T\in\mathfrak{DT}_{n-1,j}$ and $b\in\I_{j-k}$, hence $\rho^{-1}$ is well-defined. It is routine to check step-by-step, that $\rho$ and $\rho^{-1}$ are indeed inverse to each other. An example of applying the inverse map $\rho^{-1}$ can be found in Figure \ref{di-sk--2}. The proof is now completed.
\end{proof}

\begin{figure}[p]
\begin{tikzpicture}[scale=0.27]
\draw[-] (2,49) to (3,50);
\draw[-] (4,51) to (5,52);
\draw[-] (6,53) to (7,54);
\draw[-] (6,56) to (7,55);
\draw[-] (6,57) to (7,58);
\draw[-] (6,60) to (7,59);
\draw[-] (6,61) to (7,62);
\draw[-] (8,54) to (9,53);
\draw[-] (8,62) to (9,61);
\draw[-] (10,49) to (11,50);
\draw[-] (10,52) to (11,51);
\draw[-] (12,50) to (13,49);

\node at (1.5,48.5) {$\oplus$};
\node at (3.5,50.5) {$\oplus$};
\node at (5.5,52.5) {$\oplus$};
\node at (5.5,56.5) {$\oplus$};
\node at (5.5,60.5) {$\oplus$};
\node at (7.5,54.5) {$\ominus$};
\node at (7.5,58.5) {$\ominus$};
\node at (7.5,62.5) {$\ominus$};
\node at (9.5,48.5) {$\oplus$};
\node at (9.5,52.5) {$\oplus$};
\node at (9.5,60.5) {$\oplus$};
\node at (11.5,50.5) {$\ominus$};
\node at (13.5,48.5) {$\oplus$};
%in-order for nodes
\node at (1,49.6) {$3$};
\node at (3,51.6) {$4$};
\node at (5,53.6) {$5$};
\node at (5,57.6) {$2$};
\node at (5,61.6) {$1$};
\node at (8,55.6) {$6$};
\node at (8,59.6) {$11$};
\node at (7,63.6) {$12$};
\node at (9,49.6) {$8$};
\node at (10,53.6) {$7$};
\node at (10,61.6) {$13$};
\node at (12,51.6) {$9$};
\node at (14,49.6) {$10$};

\node at (2,57) {$T=$};
\node at (15,56.5) {$\rho \atop \longrightarrow$};
%leve1--1

\draw[-] (18,49) to (19,50);
\draw[-] (20,51) to (21,52);
\draw[-] (20,55) to (21,56);
\draw[-] (22,53) to (23,54);
\draw[-] (22,56) to (23,55);
\draw[-] (22,57) to (23,58);
\draw[-] (22,60) to (23,59);
\draw[-] (22,61) to (23,62);
\draw[-] (24,54) to (25,53);
\draw[-] (24,62) to (25,61);
\draw[-] (26,49) to (27,50);
\draw[-] (26,52) to (27,51);
\draw[-] (28,50) to (29,49);

\node at (17.5,48.5) {$\oplus$};
\node at (19.5,50.5) {$\oplus$};
\node at (19.5,54.5) {$\ominus$};
\node at (21.5,52.5) {$\oplus$};
\node at (21.5,56.5) {$\oplus$};
\node at (21.5,60.5) {$\oplus$};
\node at (23.5,54.5) {$\ominus$};
\node at (23.5,58.5) {$\ominus$};
\node at (23.5,62.5) {$\ominus$};
\node at (25.5,48.5) {$\oplus$};
\node at (25.5,52.5) {$\oplus$};
\node at (25.5,60.5) {$\oplus$};
\node at (27.5,50.5) {$\ominus$};
\node at (29.5,48.5) {$\oplus$};

\node at (23,47) {$b=0000$};
%leve1--2

\draw[-] (36,49) to (37,50);
\draw[-] (38,51) to (39,52);
\draw[-] (38,58) to (39,57);
\draw[-] (38,59) to (39,60);
\draw[-] (38,62) to (39,61);
\draw[-] (38,63) to (39,64);
\draw[-] (40,53) to (41,54);
\draw[-] (40,56) to (41,55);
\draw[-] (40,64) to (41,63);
\draw[-] (42,54) to (43,53);
\draw[-] (44,49) to (45,50);
\draw[-] (44,52) to (45,51);
\draw[-] (46,50) to (47,49);

\node at (35.5,48.5) {$\oplus$};
\node at (37.5,50.5) {$\oplus$};
\node at (37.5,58.5) {$\ominus$};
\node at (37.5,62.5) {$\oplus$};
\node at (39.5,52.5) {$\oplus$};
\node at (39.5,56.5) {$\oplus$};
\node at (39.5,60.5) {$\ominus$};
\node at (39.5,64.5) {$\ominus$};
\node at (41.5,54.5) {$\ominus$};
\node at (41.5,62.5) {$\oplus$};
\node at (43.5,48.5) {$\oplus$};
\node at (43.5,52.5) {$\oplus$};
\node at (45.5,50.5) {$\ominus$};
\node at (47.5,48.5) {$\oplus$};

\node at (41,47) {$b=0001$};
%level--3

\draw[-] (4,27) to (5,28);
\draw[-] (6,29) to (7,30);
\draw[-] (6,32) to (7,31);
\draw[-] (6,33) to (7,34);
\draw[-] (6,36) to (7,35);
\draw[-] (6,37) to (7,38);
\draw[-] (6,40) to (7,39);
\draw[-] (6,41) to (7,42);
\draw[-] (8,30) to (9,29);
\draw[-] (8,42) to (9,41);
\draw[-] (10,25) to (11,26);
\draw[-] (10,28) to (11,27);
\draw[-] (12,26) to (13,25);

\node at (3.5,26.5) {$\oplus$};
\node at (5.5,28.5) {$\oplus$};
\node at (5.5,32.5) {$\oplus$};
\node at (5.5,36.5) {$\ominus$};
\node at (5.5,40.5) {$\oplus$};
\node at (7.5,30.5) {$\ominus$};
\node at (7.5,34.5) {$\oplus$};
\node at (7.5,38.5) {$\ominus$};
\node at (7.5,42.5) {$\ominus$};
\node at (9.5,24.5) {$\oplus$};
\node at (9.5,28.5) {$\oplus$};
\node at (9.5,40.5) {$\oplus$};
\node at (11.5,26.5) {$\ominus$};
\node at (13.5,24.5) {$\oplus$};

\node at (7,23) {$b=0010$};
%leve2--1

\draw[-] (20,27) to (21,28);
\draw[-] (22,29) to (23,30);
\draw[-] (22,32) to (23,31);
\draw[-] (22,33) to (23,34);
\draw[-] (22,36) to (23,35);
\draw[-] (22,37) to (23,38);
\draw[-] (22,40) to (23,39);
\draw[-] (22,41) to (23,42);
\draw[-] (24,30) to (25,29);
\draw[-] (24,42) to (25,41);
\draw[-] (26,25) to (27,26);
\draw[-] (26,28) to (27,27);
\draw[-] (28,26) to (29,25);

\node at (19.5,26.5) {$\oplus$};
\node at (21.5,28.5) {$\oplus$};
\node at (21.5,32.5) {$\ominus$};
\node at (21.5,36.5) {$\ominus$};
\node at (21.5,40.5) {$\oplus$};
\node at (23.5,30.5) {$\oplus$};
\node at (23.5,34.5) {$\oplus$};
\node at (23.5,38.5) {$\ominus$};
\node at (23.5,42.5) {$\ominus$};
\node at (25.5,24.5) {$\oplus$};
\node at (25.5,28.5) {$\ominus$};
\node at (25.5,40.5) {$\oplus$};
\node at (27.5,26.5) {$\oplus$};
\node at (29.5,24.5) {$\ominus$};

\node at (24,23) {$b=0011$};
%leve2--2

\draw[-] (36,29) to (37,30);
\draw[-] (36,32) to (37,31);
\draw[-] (36,33) to (37,34);
\draw[-] (38,30) to (39,29);
\draw[-] (38,35) to (39,36);
\draw[-] (38,38) to (39,37);
\draw[-] (38,39) to (39,40);
\draw[-] (38,42) to (39,41);
\draw[-] (38,43) to (39,44);
\draw[-] (40,25) to (41,26);
\draw[-] (40,28) to (41,27);
\draw[-] (40,44) to (41,43);
\draw[-] (42,26) to (43,25);

\node at (35.5,28.5) {$\oplus$};
\node at (35.5,32.5) {$\oplus$};
\node at (37.5,30.5) {$\ominus$};
\node at (37.5,34.5) {$\oplus$};
\node at (37.5,38.5) {$\ominus$};
\node at (37.5,42.5) {$\oplus$};
\node at (39.5,24.5) {$\oplus$};
\node at (39.5,28.5) {$\oplus$};
\node at (39.5,36.5) {$\oplus$};
\node at (39.5,40.5) {$\ominus$};
\node at (39.5,44.5) {$\ominus$};
\node at (41.5,26.5) {$\ominus$};
\node at (41.5,42.5) {$\oplus$};
\node at (43.5,24.5) {$\oplus$};

\node at (40,23) {$b=0100$};
%leve2--3

\draw[-] (4,5) to (5,6);
\draw[-] (4,8) to (5,7);
\draw[-] (4,9) to (5,10);
\draw[-] (6,6) to (7,5);
\draw[-] (6,11) to (7,12);
\draw[-] (6,14) to (7,13);
\draw[-] (6,15) to (7,16);
\draw[-] (6,18) to (7,17);
\draw[-] (6,19) to (7,20);
\draw[-] (8,1) to (9,2);
\draw[-] (8,4) to (9,3);
\draw[-] (8,20) to (9,19);
\draw[-] (10,2) to (11,1);

\node at (3.5,4.5) {$\oplus$};
\node at (3.5,8.5) {$\ominus$};
\node at (5.5,6.5) {$\oplus$};
\node at (5.5,10.5) {$\oplus$};
\node at (5.5,14.5) {$\ominus$};
\node at (5.5,18.5) {$\oplus$};
\node at (7.5,0.5) {$\oplus$};
\node at (7.5,4.5) {$\ominus$};
\node at (7.5,12.5) {$\oplus$};
\node at (7.5,16.5) {$\ominus$};
\node at (7.5,20.5) {$\ominus$};
\node at (9.5,2.5) {$\oplus$};
\node at (9.5,18.5) {$\oplus$};
\node at (11.5,0.5) {$\ominus$};

\node at (7,-1) {$b=0110$};
%leve3--1

\draw[-] (20,5) to (21,6);
\draw[-] (20,8) to (21,7);
\draw[-] (20,9) to (21,10);
\draw[-] (22,6) to (23,5);
\draw[-] (22,11) to (23,12);
\draw[-] (22,14) to (23,13);
\draw[-] (22,15) to (23,16);
\draw[-] (22,18) to (23,17);
\draw[-] (22,19) to (23,20);
\draw[-] (24,1) to (25,2);
\draw[-] (24,4) to (25,3);
\draw[-] (24,20) to (25,19);
\draw[-] (26,2) to (27,1);

\node at (19.5,4.5) {$\oplus$};
\node at (19.5,8.5) {$\oplus$};
\node at (21.5,6.5) {$\ominus$};
\node at (21.5,10.5) {$\ominus$};
\node at (21.5,14.5) {$\ominus$};
\node at (21.5,18.5) {$\oplus$};
\node at (23.5,0.5) {$\oplus$};
\node at (23.5,4.5) {$\oplus$};
\node at (23.5,12.5) {$\oplus$};
\node at (23.5,16.5) {$\ominus$};
\node at (23.5,20.5) {$\ominus$};
\node at (25.5,2.5) {$\ominus$};
\node at (25.5,18.5) {$\oplus$};
\node at (27.5,0.5) {$\oplus$};

\node at (24,-1) {$b=0101$};
%leve3--2

\draw[-] (36,5) to (37,6);
\draw[-] (36,8) to (37,7);
\draw[-] (36,9) to (37,10);
\draw[-] (38,6) to (39,5);
\draw[-] (38,11) to (39,12);
\draw[-] (38,14) to (39,13);
\draw[-] (38,15) to (39,16);
\draw[-] (38,18) to (39,17);
\draw[-] (38,19) to (39,20);
\draw[-] (40,1) to (41,2);
\draw[-] (40,4) to (41,3);
\draw[-] (40,20) to (41,19);
\draw[-] (42,2) to (43,1);

\node at (35.5,4.5) {$\oplus$};
\node at (35.5,8.5) {$\ominus$};
\node at (37.5,6.5) {$\oplus$};
\node at (37.5,10.5) {$\ominus$};
\node at (37.5,14.5) {$\ominus$};
\node at (37.5,18.5) {$\oplus$};
\node at (39.5,0.5) {$\oplus$};
\node at (39.5,4.5) {$\ominus$};
\node at (39.5,12.5) {$\oplus$};
\node at (39.5,16.5) {$\ominus$};
\node at (39.5,20.5) {$\ominus$};
\node at (41.5,2.5) {$\oplus$};
\node at (41.5,18.5) {$\oplus$};
\node at (43.5,0.5) {$\ominus$};

\node at (40,-1) {$b=0111$};
%leve3--3
\end{tikzpicture}
\end{figure}

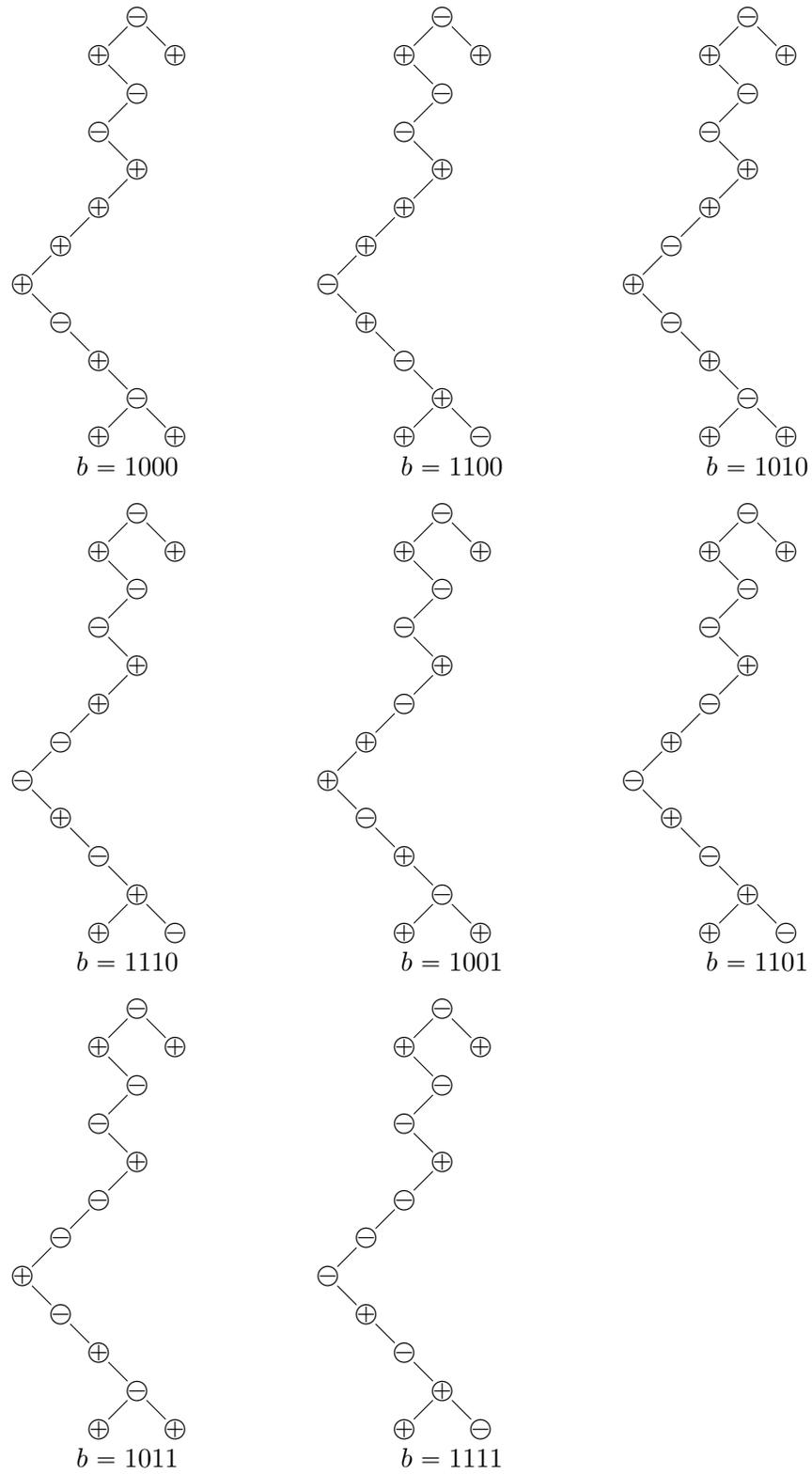
\begin{figure}
\begin{tikzpicture}[scale=0.27]
\draw[-] (2,60) to (3,59);
\draw[-] (2,61) to (3,62);
\draw[-] (4,58) to (5,57);
\draw[-] (4,63) to (5,64);
\draw[-] (6,53) to (7,54);
\draw[-] (6,56) to (7,55);
\draw[-] (6,65) to (7,66);
\draw[-] (6,68) to (7,67);
\draw[-] (6,69) to (7,70);
\draw[-] (6,72) to (7,71);
\draw[-] (6,73) to (7,74);
\draw[-] (8,54) to (9,53);
\draw[-] (8,74) to (9,73);

\node at (1.5,60.5) {$\oplus$};
\node at (3.5,58.5) {$\ominus$};
\node at (3.5,62.5) {$\oplus$};
\node at (5.5,52.5) {$\oplus$};
\node at (5.5,56.5) {$\oplus$};
\node at (5.5,64.5) {$\oplus$};
\node at (5.5,68.5) {$\ominus$};
\node at (5.5,72.5) {$\oplus$};
\node at (7.5,54.5) {$\ominus$};
\node at (7.5,66.5) {$\oplus$};
\node at (7.5,70.5) {$\ominus$};
\node at (7.5,74.5) {$\ominus$};
\node at (9.5,52.5) {$\oplus$};
\node at (9.5,72.5) {$\oplus$};

\node at (7,51) {$b=1000$};
%leve4--1

\draw[-] (18,60) to (19,59);
\draw[-] (18,61) to (19,62);
\draw[-] (20,58) to (21,57);
\draw[-] (20,63) to (21,64);
\draw[-] (22,53) to (23,54);
\draw[-] (22,56) to (23,55);
\draw[-] (22,65) to (23,66);
\draw[-] (22,68) to (23,67);
\draw[-] (22,69) to (23,70);
\draw[-] (22,72) to (23,71);
\draw[-] (22,73) to (23,74);
\draw[-] (24,54) to (25,53);
\draw[-] (24,74) to (25,73);

\node at (17.5,60.5) {$\ominus$};
\node at (19.5,58.5) {$\oplus$};
\node at (19.5,62.5) {$\oplus$};
\node at (21.5,52.5) {$\oplus$};
\node at (21.5,56.5) {$\ominus$};
\node at (21.5,64.5) {$\oplus$};
\node at (21.5,68.5) {$\ominus$};
\node at (21.5,72.5) {$\oplus$};
\node at (23.5,54.5) {$\oplus$};
\node at (23.5,66.5) {$\oplus$};
\node at (23.5,70.5) {$\ominus$};
\node at (23.5,74.5) {$\ominus$};
\node at (25.5,52.5) {$\ominus$};
\node at (25.5,72.5) {$\oplus$};

\node at (24,51) {$b=1100$};
%leve4--2

\draw[-] (34,60) to (35,59);
\draw[-] (34,61) to (35,62);
\draw[-] (36,58) to (37,57);
\draw[-] (36,63) to (37,64);
\draw[-] (38,53) to (39,54);
\draw[-] (38,56) to (39,55);
\draw[-] (38,65) to (39,66);
\draw[-] (38,68) to (39,67);
\draw[-] (38,69) to (39,70);
\draw[-] (38,72) to (39,71);
\draw[-] (38,73) to (39,74);
\draw[-] (40,54) to (41,53);
\draw[-] (40,74) to (41,73);

\node at (33.5,60.5) {$\oplus$};
\node at (35.5,58.5) {$\ominus$};
\node at (35.5,62.5) {$\ominus$};
\node at (37.5,52.5) {$\oplus$};
\node at (37.5,56.5) {$\oplus$};
\node at (37.5,64.5) {$\oplus$};
\node at (37.5,68.5) {$\ominus$};
\node at (37.5,72.5) {$\oplus$};
\node at (39.5,54.5) {$\ominus$};
\node at (39.5,66.5) {$\oplus$};
\node at (39.5,70.5) {$\ominus$};
\node at (39.5,74.5) {$\ominus$};
\node at (41.5,52.5) {$\oplus$};
\node at (41.5,72.5) {$\oplus$};

\node at (40,51) {$b=1010$};
%leve4--3

\draw[-] (2,34) to (3,33);
\draw[-] (2,35) to (3,36);
\draw[-] (4,32) to (5,31);
\draw[-] (4,37) to (5,38);
\draw[-] (6,27) to (7,28);
\draw[-] (6,30) to (7,29);
\draw[-] (6,39) to (7,40);
\draw[-] (6,42) to (7,41);
\draw[-] (6,43) to (7,44);
\draw[-] (6,46) to (7,45);
\draw[-] (6,47) to (7,48);
\draw[-] (8,28) to (9,27);
\draw[-] (8,48) to (9,47);

\node at (1.5,34.5) {$\ominus$};
\node at (3.5,32.5) {$\oplus$};
\node at (3.5,36.5) {$\ominus$};
\node at (5.5,26.5) {$\oplus$};
\node at (5.5,30.5) {$\ominus$};
\node at (5.5,38.5) {$\oplus$};
\node at (5.5,42.5) {$\ominus$};
\node at (5.5,46.5) {$\oplus$};
\node at (7.5,28.5) {$\oplus$};
\node at (7.5,40.5) {$\oplus$};
\node at (7.5,44.5) {$\ominus$};
\node at (7.5,48.5) {$\ominus$};
\node at (9.5,26.5) {$\ominus$};
\node at (9.5,46.5) {$\oplus$};

\node at (7,25) {$b=1110$};
%leve5--1

\draw[-] (18,34) to (19,33);
\draw[-] (18,35) to (19,36);
\draw[-] (20,32) to (21,31);
\draw[-] (20,37) to (21,38);
\draw[-] (22,27) to (23,28);
\draw[-] (22,30) to (23,29);
\draw[-] (22,39) to (23,40);
\draw[-] (22,42) to (23,41);
\draw[-] (22,43) to (23,44);
\draw[-] (22,46) to (23,45);
\draw[-] (22,47) to (23,48);
\draw[-] (24,28) to (25,27);
\draw[-] (24,48) to (25,47);

\node at (17.5,34.5) {$\oplus$};
\node at (19.5,32.5) {$\ominus$};
\node at (19.5,36.5) {$\oplus$};
\node at (21.5,26.5) {$\oplus$};
\node at (21.5,30.5) {$\oplus$};
\node at (21.5,38.5) {$\ominus$};
\node at (21.5,42.5) {$\ominus$};
\node at (21.5,46.5) {$\oplus$};
\node at (23.5,28.5) {$\ominus$};
\node at (23.5,40.5) {$\oplus$};
\node at (23.5,44.5) {$\ominus$};
\node at (23.5,48.5) {$\ominus$};
\node at (25.5,26.5) {$\oplus$};
\node at (25.5,46.5) {$\oplus$};

\node at (24,25) {$b=1001$};
%leve5--2

\draw[-] (34,34) to (35,33);
\draw[-] (34,35) to (35,36);
\draw[-] (36,32) to (37,31);
\draw[-] (36,37) to (37,38);
\draw[-] (38,27) to (39,28);
\draw[-] (38,30) to (39,29);
\draw[-] (38,39) to (39,40);
\draw[-] (38,42) to (39,41);
\draw[-] (38,43) to (39,44);
\draw[-] (38,46) to (39,45);
\draw[-] (38,47) to (39,48);
\draw[-] (40,28) to (41,27);
\draw[-] (40,48) to (41,47);

\node at (33.5,34.5) {$\ominus$};
\node at (35.5,32.5) {$\oplus$};
\node at (35.5,36.5) {$\oplus$};
\node at (37.5,26.5) {$\oplus$};
\node at (37.5,30.5) {$\ominus$};
\node at (37.5,38.5) {$\ominus$};
\node at (37.5,42.5) {$\ominus$};
\node at (37.5,46.5) {$\oplus$};
\node at (39.5,28.5) {$\oplus$};
\node at (39.5,40.5) {$\oplus$};
\node at (39.5,44.5) {$\ominus$};
\node at (39.5,48.5) {$\ominus$};
\node at (41.5,26.5) {$\ominus$};
\node at (41.5,46.5) {$\oplus$};

\node at (40,25) {$b=1101$};
%leve5--3

\draw[-] (2,8) to (3,7);
\draw[-] (2,9) to (3,10);
\draw[-] (4,6) to (5,5);
\draw[-] (4,11) to (5,12);
\draw[-] (6,1) to (7,2);
\draw[-] (6,4) to (7,3);
\draw[-] (6,13) to (7,14);
\draw[-] (6,16) to (7,15);
\draw[-] (6,17) to (7,18);
\draw[-] (6,20) to (7,19);
\draw[-] (6,21) to (7,22);
\draw[-] (8,2) to (9,1);
\draw[-] (8,22) to (9,21);

\node at (1.5,8.5) {$\oplus$};
\node at (3.5,6.5) {$\ominus$};
\node at (3.5,10.5) {$\ominus$};
\node at (5.5,0.5) {$\oplus$};
\node at (5.5,4.5) {$\oplus$};
\node at (5.5,12.5) {$\ominus$};
\node at (5.5,16.5) {$\ominus$};
\node at (5.5,20.5) {$\oplus$};
\node at (7.5,2.5) {$\ominus$};
\node at (7.5,14.5) {$\oplus$};
\node at (7.5,18.5) {$\ominus$};
\node at (7.5,22.5) {$\ominus$};
\node at (9.5,0.5) {$\oplus$};
\node at (9.5,20.5) {$\oplus$};

\node at (7,-1) {$b=1011$};
%leve6--1

\draw[-] (18,8) to (19,7);
\draw[-] (18,9) to (19,10);
\draw[-] (20,6) to (21,5);
\draw[-] (20,11) to (21,12);
\draw[-] (22,1) to (23,2);
\draw[-] (22,4) to (23,3);
\draw[-] (22,13) to (23,14);
\draw[-] (22,16) to (23,15);
\draw[-] (22,17) to (23,18);
\draw[-] (22,20) to (23,19);
\draw[-] (22,21) to (23,22);
\draw[-] (24,2) to (25,1);
\draw[-] (24,22) to (25,21);

\node at (17.5,8.5) {$\ominus$};
\node at (19.5,6.5) {$\oplus$};
\node at (19.5,10.5) {$\ominus$};
\node at (21.5,0.5) {$\oplus$};
\node at (21.5,4.5) {$\ominus$};
\node at (21.5,12.5) {$\ominus$};
\node at (21.5,16.5) {$\ominus$};
\node at (21.5,20.5) {$\oplus$};
\node at (23.5,2.5) {$\oplus$};
\node at (23.5,14.5) {$\oplus$};
\node at (23.5,18.5) {$\ominus$};
\node at (23.5,22.5) {$\ominus$};
\node at (25.5,0.5) {$\ominus$};
\node at (25.5,20.5) {$\oplus$};

\node at (24,-1) {$b=1111$};
%leve6--2
\end{tikzpicture}
\caption{Images of $(T,b)$ under the map $\rho$.}\label{di-sk--1}
\end{figure}

\begin{figure}
\begin{tikzpicture}[scale=0.27]
\draw[-] (2,3) to (3,4);
\draw[-] (4,5) to (5,6);
\draw[-] (6,7) to (7,8);
\draw[-] (6,10) to (7,9);
\draw[-] (6,11) to (7,12);
\draw[-] (8,8) to (9,7);
\draw[-] (8,12) to (9,11);
\draw[-] (10,6) to (11,5);
\draw[-] (12,4) to (13,3);
\draw[-] (8,5) to (9,6);
\draw[-] (6,3) to (7,4);
\draw[-] (6,2) to (7,1);

\node at (3.5,4.5) {$\oplus$};
\node at (1.5,2.5) {$\oplus$};
\node at (5.5,6.5) {$\oplus$};
\node at (5.5,10.5) {$\oplus$};
\node at (7.5,8.5) {$\ominus$};
\node at (7.5,12.5) {$\ominus$};
\node at (9.5,6.5) {$\oplus$};
\node at (9.5,10.5) {$\oplus$};
\node at (11.5,4.5) {$\ominus$};
\node at (13.5,2.5) {$\oplus$};
\node at (7.5,4.5) {$\oplus$};
\node at (5.5,2.5) {$\ominus$};
\node at (7.5,.5) {$\oplus$};

\node at (1,3.6) {$2$};
\node at (3,5.6) {$3$};
\node at (5,7.6) {$4$};
\node at (5,11.6) {$1$};
\node at (8,13.6) {$12$};
\node at (8,9.6) {$5$};
\node at (8,5.6) {$8$};
\node at (8,1.6) {$7$};
\node at (5,3.6) {$6$};
\node at (10,7.6) {$9$};
\node at (10,11.6) {$13$};
\node at (12,5.6) {$10$};
\node at (14,3.6) {$11$};

\node at (0,9) {$S=$};
\node at (17.5,9) {$\rho^{-1} \atop \longrightarrow$};

%%% preimage (T,b)
\node at (21,9) {$T=$};
\node at (22,4.5) {$\oplus$};
\node at (24,6.5) {$\oplus$};
\node at (26,8.5) {$\oplus$};
\node at (28,10.5) {$\oplus$};
\node at (30,12.5) {$\oplus$};
\node at (32,14.5) {$\ominus$};
\node at (30,8.5) {$\oplus$};
\node at (32,6.5) {$\ominus$};
\node at (34,4.5) {$\oplus$};
\node at (32,10.5) {$\ominus$};
\node at (28,6.5) {$\oplus$};
\node at (34,12.5) {$\oplus$};
\node at (41,9) {$b=110$};
\node at (37,8.5) {,};

\draw[-] (22.5,5) to (23.5,6);
\draw[-] (24.5,7) to (25.5,8);
\draw[-] (26.5,9) to (27.5,10);
\draw[-] (28.5,11) to (29.5,12);
\draw[-] (30.5,13) to (31.5,14);
\draw[-] (30.5,12) to (31.5,11);
\draw[-] (32.5,14) to (33.5,13);
\draw[-] (28.5,7) to (29.5,8);
\draw[-] (30.5,9) to (31.5,10);
\draw[-] (30.5,8) to (31.5,7);
\draw[-] (32.5,6) to (33.5,5);
\end{tikzpicture}
\caption{The preimage of $S$ under the map $\rho$.} \label{di-sk--2}
\end{figure}
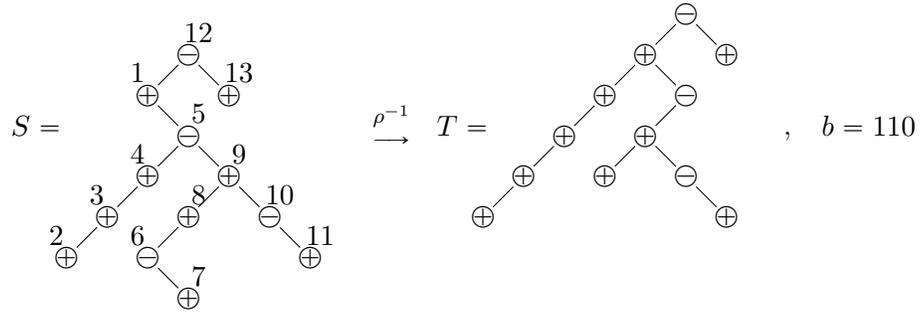

%%%%%%%%%%%%%%%%%%%%%
\section{Final remarks}
%%%%%%%%%%%%%%%%%%%%%
Several problems arising from this work merit further study.

First, recall the algebraic proofs of Theorems~\ref{thm:Sch} and \ref{thm:litSch} via Riordan arrays that we have outlined in the Introduction. This approach is fully applicable to the enumeration of Schr\"oder paths refined by two extra parameters, one of which keeps track of the number of horizontals at height $0$, while the other one counts the remaining horizontals.

For any $p\in\R$, we use $h_0(p)$, $h(p)$, and $|p|$ to denote the number of horizontals at height $0$, the number of horizontals at positive height, and the semi-length of $p$, respectively. Then the generating function of Schr\"oder paths with precisely $k$ hills is given by
\begin{align*}
x^kg(u,v;x)^{k+1} &=g(u,v;x) (xg(u,v;x))^k,
\end{align*}
where
\begin{align*}
g(u,v;x) &:=\sum_{p\in\widetilde{\R}}u^{h_0(p)}v^{h(p)}x^{|p|}=\frac{2}{1+(2+v-2u)x+\sqrt{1-(4+2v)x+v^2x^2}}
\end{align*}
is the generating function of Schr\"oder paths without hills. Therefore $(g(u,v;x),xg(u,v;x))$ is now a $(u,v)$-weighted Riordan array of Bell type. We can analogously solve for its $A$- and $Z$-sequences:
\begin{align}
A(u,v;x) &= \frac{(1-x+ux)(1-x+ux-vx)}{1-2x+ux-vx}=1+ux+\sum_{j\ge 0}(1+v)(2-u+v)^jx^{j+2},\label{uvA}\\
Z(u,v;x) &= \frac{1}{x}(A(u,v;x)-1)=u+\sum_{j\ge 0}(1+v)(2-u+v)^jx^{j+1}.\label{uvZ}
\end{align}

Thanks to the addition of parameters $u$ and $v$, this $(g(u,v;x),xg(u,v;x))$ encapsulates quite a few Riordan arrays as special cases. See Table~\ref{uv-specialization} for some examples. But the somewhat unexpected presence of the term $-u$ in both \eqref{uvA} and \eqref{uvZ} makes it an intriguing problem to find bijective proofs of the following recurrences implied by them. 
\begin{theorem}
Let $r_{u,v}(0,0)=1$, then we have for $n\ge 1$,
\begin{align}
r_{u,v}(n,0)&=ur_{u,v}(n-1,0)+(1+v)\sum\limits_{j=0}^{n-2}(2-u+v)^{j}r_{u,v}(n-1,j+1),\label{uvZ-rec}
\end{align}
and for $1\le k\le n$,
\begin{align}
r_{u,v}(n,k)&=r_{u,v}(n-1,k-1)+ur_{u,v}(n-1,k)+(1+v)\sum\limits_{j=k}^{n-2}(2-u+v)^{j-k}r_{u,v}(n-1,j+1),\label{uvA-rec}
\end{align}
where $r_{u,v}(n,k)$ is a polynomial in $u,v$ and also the coefficient of $x^n$ in $x^kg(u,v;x)^{k+1}$. Combinatorially, it is the $(u,v)$-weighted counting of the Schr\"oder paths of length $2n$ with $k$ hills.
\end{theorem}

\begin{table}
{\small
\begin{tabular}{cccc}
\toprule
$(u,v)$ & $(g(x),f(x))$ & in OEIS? & row sums\\
\midrule
&&&\\
$(1,1)$ & $(\dfrac{2}{1+x+\sqrt{1-6x+x^2}},\dfrac{2x}{1+x+\sqrt{1-6x+x^2}})$ & \href{http://oeis.org/A104219}{A104219} & $1,2,6,22,90,\cdots$ \href{http://oeis.org/A006318}{A006318}\\
&&&\\
$(0,1)$ & $(\dfrac{2}{1+3x+\sqrt{1-6x+x^2}},\dfrac{2x}{1+3x+\sqrt{1-6x+x^2}})$ & \href{http://oeis.org/A114709}{A114709} & $1,1,3,11,45,\cdots$ \href{http://oeis.org/A001003}{A001003}\\
&&&\\
$(-1,1)$ & $(\dfrac{2}{1+5x+\sqrt{1-6x+x^2}},\dfrac{2x}{1+5x+\sqrt{1-6x+x^2}})$ & new & $1,0,2,6,26,\cdots$ \href{http://oeis.org/A114710}{A114710}\\
&&&\\
$(2,1)$ & $(\dfrac{2}{1-x+\sqrt{1-6x+x^2}},\dfrac{2x}{1-x+\sqrt{1-6x+x^2}})$ & \href{http://oeis.org/A080247}{A080247} & $1,3,11,45,197,\cdots$ \href{http://oeis.org/A001003}{A001003}\\
&&&\\
$(1,0)$ & $(\dfrac{2}{1+\sqrt{1-4x}},\dfrac{2x}{1+\sqrt{1-4x}})$ & \href{http://oeis.org/A033184}{A033184} & $1,2,5,14,42,\cdots$ \href{http://oeis.org/A000108}{A000108}\\
&&&\\
$(0,0)$ & $(\dfrac{2}{1+2x+\sqrt{1-4x}},\dfrac{2x}{1+2x+\sqrt{1-4x}})$ & \href{http://oeis.org/A065600}{A065600} & $1,1,2,5,14,\cdots$ \href{http://oeis.org/A000108}{A000108}\\
&&&\\
$(-1,0)$ & $(\dfrac{2}{1+4x+\sqrt{1-4x}},\dfrac{2x}{1+4x+\sqrt{1-4x}})$ & new & $1,0,1,2,6,\cdots$ \href{http://oeis.org/A000957}{A000957}\\
&&&\\
$(2,0)$ & $(\dfrac{2}{1-2x+\sqrt{1-4x}},\dfrac{2x}{1-2x+\sqrt{1-4x}})$ & \href{http://oeis.org/A039598}{A039598} & $1,3,10,35,126,\cdots$ \href{http://oeis.org/A001700}{A001700}\\
&&&\\
$(1,-1)$ & $(\dfrac{1}{1-x},\dfrac{x}{1-x})$ & \href{http://oeis.org/A007318}{A007318} & $1,2,4,8,16,\cdots$ \href{http://oeis.org/A000079}{A000079}\\
&&&\\
$(0,-1)$ & $(1,x)$ & \href{http://oeis.org/A010054}{A010054} & $1,1,1,1,1,\cdots$ \href{http://oeis.org/A000012}{A000012}\\
&&&\\
$(-1,-1)$ & $(\dfrac{1}{1+x},\dfrac{x}{1+x})$ & \href{http://oeis.org/A130595}{A130595} & $1,0,0,0,0,\cdots$ \href{http://oeis.org/A000007}{A000007}\\
&&&\\
$(2,-1)$ & $(\dfrac{1}{1-2x},\dfrac{x}{1-2x})$ & \href{http://oeis.org/A038207}{A038207} & $1,3,9,27,81,\cdots$ \href{http://oeis.org/A000244}{A000244}\\
&&&\\
$(1,2)$ & $(\dfrac{2}{1+2x+\sqrt{1-8x+4x^2}},\dfrac{2x}{1+2x+\sqrt{1-8x+4x^2}})$ & new & $1,2,7,32,166,\cdots$ \href{http://oeis.org/A108524}{A108524}\\
&&&\\
$(0,2)$ & $(\dfrac{2}{1+4x+\sqrt{1-8x+4x^2}},\dfrac{2x}{1+4x+\sqrt{1-8x+4x^2}})$ & new & $1,1,4,19,100,\cdots$ \href{http://oeis.org/A007564}{A007564}\\
&&&\\
$(-1,2)$ & $(\dfrac{2}{1+6x+\sqrt{1-8x+4x^2}},\dfrac{2x}{1+6x+\sqrt{1-8x+4x^2}})$ & new & $1,0,3,12,66,\cdots$ new\\
&&&\\
$(2,2)$ & $(\dfrac{2}{1+\sqrt{1-8x+4x^2}},\dfrac{2x}{1+\sqrt{1-8x+4x^2}})$ & new & $1,3,12,57,300,\cdots$ \href{http://oeis.org/A047891}{A047891}\\
&&&\\
\bottomrule
\end{tabular}
}
\caption{Specializations of $(u,v)$, Riordan arrays, and their row sums.}
\label{uv-specialization}
\end{table}

Next, we make some comments on the part (ii) of Problem \ref{find stat}. In Sections~\ref{sec:comb pf} and \ref{sec:comb pf sep}, we have constructed bijections to interpret the same recurrences \eqref{sch-Z}--\eqref{sch-A} and \eqref{sep-Z}--\eqref{sep-A}, via Schr\"oder paths and di-sk trees, respectively. Moreover, the set of binary sequences $\B_k$ and the set of binary numbers $\I_k$ are in natural bijection with each other. This means we get for free, a recursively defined bijection between Schr\"oder paths of length $2n$ and di-sk trees with $n$ nodes, such that the number of hills is sent to the index of the first $\ominus$-node.
$$\R_{n,k}\setminus\R_{n-1,k-1}\stackrel{\Phi}{\longleftrightarrow}\bigcup_{j=k}^{n-1}\R_{n-1,j}\times\B_{j-k}\longleftrightarrow\bigcup_{j=k}^{n-1}\DT_{n,j+1}\times\I_{j-k}\stackrel{\rho}{\longleftrightarrow}\DT_{n+1,k+1}\setminus\DT_{n,k}$$
When combined with the bijection $\eta$ in Theorem~\ref{FLZ-bij}, this already gives an answer to Problem~\ref{find stat} (ii). Although finding a more direct bijection is still appealing.

On the other hand, with the aid of certain leaf-marked plane rooted trees, Claesson, Kitaev and Steingr\'imsson constructed two bijections between Schr\"oder paths of length $2n$ and separable permutations of length $n+1$. The second of their bijections pays special attention to the number of hills.

\begin{theorem}[cf. Theorem~2.2.48 in \cite{Kit}]\label{thm:CKS}
There is a bijection between the separable permutations
in $\mathfrak{S}_{n+1}(2413,3142)$ and the Schr\"oder paths of length $2n$ such that the statistic $\comp$ on permutations corresponds to
$\comp_s$ on paths.
\end{theorem}
Here $\comp_s(p)=1+h_0(p)$ for each Schr\"oder path $p$, and the definition of $\comp$ is given below.
\begin{Def}
For any permutation $\pi$, $\comp(\pi)$ is defined as the number of ways to factor $\pi=\sigma\tau$, so that each letter in the non-empty $\sigma$ is smaller than any letter in $\tau$, and $\tau$ is allowed to be empty.
\end{Def}

\begin{table}[h]
\centering 
\begin{tabular}{VcVc|c|c|c|c|c|c|cV}
\Xhline{2pt}
$\pi$ & $123$ & $132$ & $213$ & $231$ & $312$ & $321$ & $2413$ & $3142$\\
\Xhline{1pt}
$\iar$ & $3$ & $2$ & $1$ & $2$ & $1$ & $1$ & $2$ & $1$\\
\hline
$\comp$ & $3$ & $2$ & $2$ & $1$ & $1$ & $1$ & $1$ & $1$\\
\Xhline{2pt}
\end{tabular}
\vspace{3mm}
\caption{Values of $\iar$ vs. $\comp$ for eight permutations.}
\label{iar-comp}
\end{table}

We compare the values of $\iar$ and $\comp$ in Table~\ref{iar-comp} for eight permutations. This table presages the following corollary, which follows from Theorems~\ref{Sep-iar} and \ref{thm:CKS}. We can also see from the last two columns in Table~\ref{iar-comp} that in general, $\iar$ and $\comp$ are {\bf not} equidistributed over the entire symmetric group.

\begin{corollary}
For all $n\ge 1$, the two permutation statistics $\iar$ and $\comp$ are equidistributed on $\mathfrak{S}_n(2413,3142)$.
\end{corollary}

Finally, it is worthwhile to study the newly introduced permutation statistic $\iar$ for its own sake. In our future work \cite{FLW}, we plan to address the Wilf-equivalence problem for various classes of pattern-avoiding permutations, with $\iar$ in mind.

\section*{Acknowledgement}
We thank Sergey Kitaev and Zhicong Lin for the useful discussions. Both authors were supported by the Fundamental Research Funds for the Central Universities (No.~2018CDXYST0024).

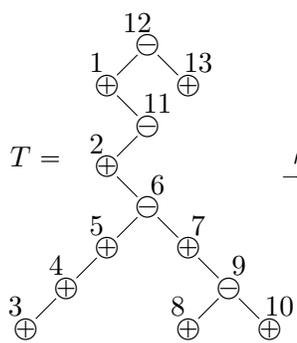
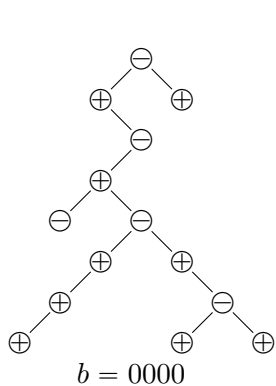
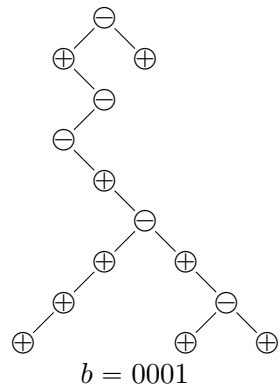
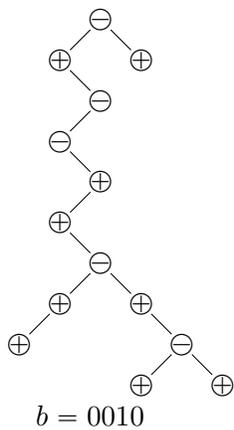
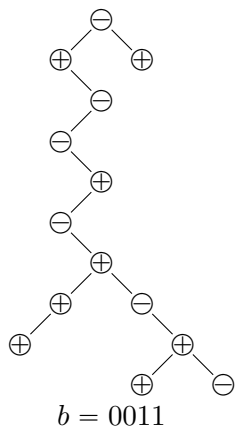
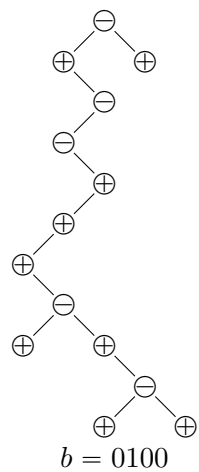
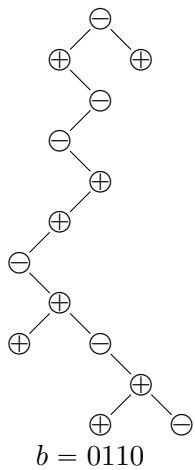
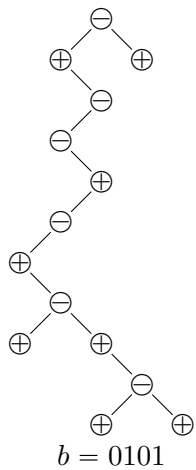
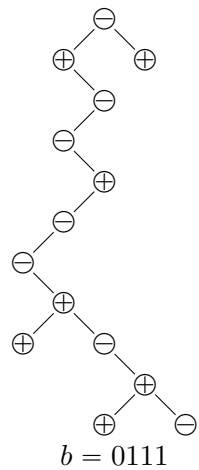
\end{document}